\documentclass[a4paper,11pt]{amsart}

\usepackage{amsmath,amssymb,amsfonts,graphicx,color,fancyhdr}
\usepackage{amsrefs}
\usepackage{float} 
\usepackage[latin1]{inputenc}
\usepackage{mathtools} 
\usepackage[hang,flushmargin]{footmisc}
\usepackage[new]{old-arrows}
\usepackage[colorlinks=true, pdfstartview=FitV, linkcolor=blue, citecolor=blue, urlcolor=blue]{hyperref}
\usepackage{tikz}

\newtheorem{theorem}{Theorem}[section]

\newtheorem{example}[theorem]{Example}
\newtheorem{lemma}[theorem]{Lemma}

\newtheorem{corollary}[theorem]{Corollary}

\newcommand{\mc}[1]{\mathcal{#1}}

\newcommand{\mf}[1]{\mathfrak{#1}}

\newcommand{\prs}[1]{\left(#1\right)}
\newcommand{\prsm}[1]{\bigg(#1\bigg)}
\newcommand{\prss}[1]{(#1)}
\newcommand{\sbrks}[1]{\left[#1\right]}
\newcommand{\abs}[1]{\left|#1\right|}
\newcommand{\abss}[1]{|#1|}
\newcommand{\absm}[1]{\bigg|#1\bigg|}
\newcommand{\angs}[1]{\left\langle#1\right\rangle}
\newcommand{\angss}[1]{\langle#1\rangle}
\newcommand{\cbrks}[1]{\left\{#1\right\}}

\newcommand{\cbrksm}[1]{\bigg\{#1\bigg\}}

\DeclarePairedDelimiter{\floor}{\lfloor}{\floor}

\newcommand{\norm}[1]{\left\lVert#1\right\rVert}

\newcommand{\norms}[1]{\lVert#1\rVert}

\newcommand{\rnm}{\mathbb{R}}
\newcommand{\znm}{\mathbb{Z}}

\newcommand{\eun}[1]{\rnm^{#1}}

\newcommand{\se}{\subseteq}
\newcommand{\cl}{\colon}
\DeclareMathOperator{\sppt}{\text{sppt}}
\DeclareMathOperator{\diam}{\text{diam}}
\newcommand{\nlspace}{\\[0.3cm]}
\newcommand{\os}[1]{^{#1}}

\newcommand{\dx}[1]{\,d#1}
\newcommand{\sm}{\setminus}
\newcommand{\nf}{\infty}
\newcommand{\ol}[1]{\overline{#1}}
\DeclareMathOperator{\sign}{\text{sign}}
\newcommand{\ra}{\rightarrow}
\newcommand{\tm}{\times}
\newcommand{\fv}{^{-1}}
\newcommand{\sph}[1]{\mathbb{S}^{#1}}
\newcommand{\pd}{\partial}
\newcommand{\wh}[1]{\widehat{#1}}

\def\Xint#1{\mathchoice
	{\XXint\displaystyle\textstyle{#1}}%
	{\XXint\textstyle\scriptstyle{#1}}%
	{\XXint\scriptstyle\scriptscriptstyle{#1}}%
	{\XXint\scriptscriptstyle\scriptscriptstyle{#1}}%
	\!\int}
\def\XXint#1#2#3{{\setbox0=\hbox{$#1{#2#3}{\int}$}
		\vcenter{\hbox{$#2#3$}}\kern-.5\wd0}}

\def\dashint{\Xint-}

\newcommand{\tlac}{T_{\text{\normalfont lac}}}
\newcommand{\tful}{T_{\text{\normalfont full}}}
\newcommand{\blac}{B_{\text{\normalfont lac}}}
\newcommand{\bful}{B_{\text{\normalfont full}}}

\newcommand{\ef}[1]{e\os{#1}}
\newcommand{\fconst}{2\pi i}

\newcommand{\llac}{L_{\text{\normalfont lac}}}
\newcommand{\lful}{L_{\text{\normalfont full}}}

\DeclareMathOperator{\conv}{conv}

\allowdisplaybreaks

\author{Eyvindur Ari Palsson and Sean R.\ Sovine}
\address{
}
\email{}
\address[Eyvindur Ari Palsson and Sean R.\ Sovine]{
	Department of Mathematics\\
	Virginia Tech\\
	Blacksburg, Virginia 24061
	}
\email{\{palsson,sovine5\}@vt.edu}


\date{\today}

\thanks{The authors would like to thank Yumeng Ou for helpful discussions about the present work as well as the anonymous referee for the many helpful suggestions. The work of the first listed author was supported in part by Simons Foundation Grant \#360560.}

\begin{document}
	
\title[Sparse Bounds for Maximal Triangle and Bilinear Spherical Averaging Operators]{Sparse Bounds for Maximal Triangle and\\ Bilinear Spherical Averaging Operators}

\begin{abstract}
	We show that the method in the recent work \cite{RSS} by Roncal, Shrivastava, and Shuin can be adapted to show that certain $L^p$-improving bounds in the interior of the boundedness region for the bilinear spherical or triangle averaging operator imply sparse bounds for the corresponding lacunary maximal operator, and that $L^p$-improving bounds in the interior of the boundedness region for the single-scale maximal bilinear spherical averaging operator implies sparse bounds for the corresponding full maximal operator.
	More generally we show that the proof in \cite{RSS} applies for bilinear convolutions with compactly supported finite Borel measures that satisfy appropriate $L^p$-improving and continuity estimates. This shows that the method used by Roncal, Shrivastava, and Shuin in \cite{RSS} can be adapted obtain sparse bounds for a general class of bilinear operators that are not of product type, for a certain range of $L^p$ exponents.
\end{abstract}
\maketitle

\renewcommand{\contentsname}{\normalsize Contents}

\vspace*{-6mm}
\tableofcontents

\vspace*{-8mm}
\section{Introduction}

In this work we consider maximal operators corresponding to a bilinear convolution with a compactly supported finite Borel measure $\mu$. Specifically, we define the scale-$t$ bilinear convolution with $\mu$ for $t > 0$ by
\[
L_t(f,g)(x) := [(f\otimes g)\star \mu_t](x,x) = \int_{\eun{2d}}f(x - ty)g(x - tz)\dx{\mu(y,z)}, 
\]
and the corresponding lacunary maximal operator
\[
\llac(f,g)(x) := \sup_{j \in \znm}\abss{L_{2^j}(f,g)(x)},
\]
and single-scale and full maximal operators
\[
L_{\star, t}(f,g)(x) := \sup_{s \in [t, 2t]}\abss{L_{s}(f,g)(x)}, \quad \text{and} \quad \lful(f,g)(x) := \sup_{t > 0}\abss{L_t(f,g)(x)}. 
\]
Operators of this type have been actively studied in recent work in harmonic analysis, in particular the bilinear spherical averaging operator \cite{GebaEA, BarrionuevoEA, JeongLee} and the triangle averaging operator \cite{PS, IPS}, and have applications to Falconer-type theorems in continuous geometric combinatorics \cite{GI, GGIP, IL}. Bounds for these operators can be viewed as bilinear extensions of similar results for linear convolutions, such as the spherical averaging operator, which has been extensively studied by harmonic analysts \cite{Strichartz, Littman, Stein, Bourgain}. 

For locally integrable functions $f,g,h$ and a sparse collection $\mc{S}$ of cubes we define the sparse form
\[
\Lambda_{\mc{S}, p, q, r'}(f,g,h) = \sum_{Q \in \mc{S}}|Q|\angs{f}_{Q, p}\angs{g}_{Q, q}\angs{h}_{Q, r'},
\]
where
\[
\angs{\varphi}_{Q, t} := \prs{ \frac{1}{|Q|}\int_Q |\varphi(x)|^t\dx{x} }\os{\frac{1}{t}} = \frac{1}{|Q|\os{1/t}}\norm{\varphi 1_Q}_{L^t}.
\]
A family $\mc{S}$ of cubes in $\eun{d}$ is called $\gamma$-\textit{sparse} if there is some fixed $\gamma \in (0,1)$ such that for each $Q \in \mc{S}$ there exists $F_Q \se Q$ with $|F_Q| \geq \gamma|Q|$ and $F_Q\cap F_{Q'} = \emptyset$ for distinct $Q, Q' \in \mc{S}$. \textit{Sparse domination} is an approach to bounding operators in dual form by obtaining a \textit{sparse bound}, which for a bilinear operator $L$ looks like 
\[
\angss{L(f,g), h} \lesssim \sup_{\mc{S}}\Lambda_{\mc{S}, p, q, r'}(f,g,h),
\]
with the supremum taken over all $\gamma$-sparse collections $\mc{S}$ for some fixed $\gamma\in(0,1)$ and the implicit constant depending only on the operator $L$. As usual, $\angs{\phi,\psi}$ is just the integral of the product $\phi\psi$ on $\eun{d}$ with respect to Lebesgue measure. 

Sparse bounds are refined estimates that yield a range of weighted bounds and to which extrapolation techniques can be applied, as described in Section~\ref{secWeight}. 
Sparse domination techniques arose out of the body of work investigating quantitative weighted bounds which lead up to the proofs of the $A_2$ conjecture. An exposition of this development is contained in \cite{Pereyra}. In \cite{LaceySpherical} Michael Lacey built on an argument of Conde, Culiuc, Di Plinio, and Ou \cite{CondeEA} to establish sparse bounds for the lacunary and full spherical maximal operators, using the key observation that the spherical average and its unit-scale maximal counterpart satisfy $L^p$-continuity estimates.  

Recently, Roncal, Shrivastava, and Shuin \cite{RSS} built on the work of Lacey \cite{LaceySpherical} to obtain sparse bounds for the product-type operators
\[
\mc{M}_{\text{lac}}(f,g)(x):= \sup_{j \in \znm}\abss{\mc{A}_{2^j}(f)(x)\mc{A}_{2^j}(g)(x)} \quad \text{and} \quad \mc{M}_{\text{full}}(f,g)(x):= \sup_{t > 0}\abss{\mc{A}_{t}(f)(x)\mc{A}_{t}(g)(x)},
\]
where $\mc{A}_{t}(f)(x)=\int_{\sph{d-1}} f(x-ty)d\sigma(y)$ is the classical spherical averaging operator, which correspond to $L_t$ in the case where $\mu$ is the natural measure on the product $\sph{d-1}\tm \sph{d-1}$. 
In this work we show that the methods in \cite{RSS} can be adapted to obtain sparse bounds for maximal operators $\llac$ and $\lful$ from $L^p$-improving and continuity estimates for $L_t$ and $L_{\star, t}$, respectively, provided that these latter estimates are available for a given measure $\mu$. 

We also address two particular instances of the operator $L_t$, the bilinear spherical averaging operator and the triangle averaging operator, for which the required $L^p$-improving and continuity estimates are available. 
The bilinear spherical averaging operator of radius $t > 0$ is defined by
\[
B_t(f,g)(x) := \int_{\sph{2d-1}}f(x - ty)g(x - tz)\dx{\sigma(y,z)},
\]
corresponds to $\mu$ being the natural surface measure on $\sph{2d-1}$. The triangle averaging operator is defined by
\[
T_t(f,g)(x) := \int_{  \mf{D} }f(x - ty)g(x - tz)\dx{\mu(y,z)},
\]
where $\mf{D} = \{ (y,z) \in \eun{2d}\, \cl |y| = |z| = | y - z | = 1 \}$ and $\mu$ is the natural surface measure on $\mf{D}$ as an embedded submanifold of $\eun{2d}$.  We discuss the required $L^p$-improving and continuity estimates for these operators in Sections \ref{secKnownBounds} and \ref{secContEst} below. 

Our main result for the triangle and bilinear spherical averaging operators is:
\begin{theorem}\label{thm0}
	Let $(1/p, 1/q, 1/r)$ with $r \geq p,q$ and $r  > 1$ be in the interior of the boundedness set $\mf{S}^d$, $d\geq 2$ of $B_{t}$. Then
	\begin{equation*}
	\abs{\angss{\blac(f,g), h}} \lesssim \sup_{\mc{S}}\Lambda_{\mc{S}, p, q, r'}(f, g, h)
	\end{equation*}
	for any compactly supported bounded functions $f$, $g$ and $h$,
	where the supremum is over all sparse collections.
	
	Let $(1/p, 1/q, 1/r)$ with $r \geq p,q$ and $r  > 1$ be in the interior of the boundedness set $\mf{S}_{\star}^d$, $d\geq 4$, of $B_{\star,t}$. Then
	\begin{equation*}
	\abs{\angss{\bful(f,g), h}} \lesssim \sup_{\mc{S}}\Lambda_{\mc{S}, p, q, r'}(f, g, h)
	\end{equation*}
	for any compactly supported bounded functions $f$, $g$ and $h$,
	where the supremum is over all sparse collections.
	
	Let $(1/p, 1/q, 1/r)$ with $r \geq p,q$ and $r  > 1$ be in the interior of the boundedness set $\mf{T}^d$, $d\geq 13$ of $T_t$. Then
	\begin{equation*}
	\abs{\angss{\tlac(f,g), h}} \lesssim \sup_{\mc{S}}\Lambda_{\mc{S}, p, q, r'}(f, g, h)
	\end{equation*}
	for any compactly supported bounded functions $f$, $g$ and $h$,
	where the supremum is over all sparse collections.
\end{theorem}
\noindent In Section \ref{secKnownBounds} we describe subsets of each of sets $\mf{S}^d$, $\mf{S}_{\star}^d$, and $\mf{T}^d$ that contain indices $(p,q,r)$ for which the above theorem can be applied. The dimensional restrictions in the above theorem all come from obtaining certain continuity bounds, while the mapping properties described in Section \ref{secKnownBounds} in many cases reach dimensions below our stated thresholds.

The lacunary results in Theorem \ref{thm0} follow by application of the following abstract sparse domination theorem for operators of the type $L_t$, which is proved in Section \ref{secAbsDom}.
\begin{theorem}\label{thm01}
	Let $L_t$ be a bilinear convolution with a compactly supported finite Borel measure $\mu$ satisfying an $L^p$-improving estimate 
	\[
	L_t \cl L^p \tm L^q \ra L^r \quad \text{for some $(p,q,r)$ with } \quad \frac{1}{p} + \frac{1}{q} > \frac{1}{r}, 
	\]
	and let 
	$\mf{B}$ be the \textit{boundedness set} of all triples $(1/p, 1/q, 1/r)$ such that 
	\[
	L_t \cl L^p(\eun{d})\tm L^q(\eun{d}) \ra L^r(\eun{d}). 
	\]
	Let $\tau_y h(x) := h(x-y)$ denote translation. Suppose that $L_t$ satisfies an $L^p$-continuity estimate
	\[
	\norm{L_t((I - \tau_{y_1})f, (I - \tau_{y_2})g)}_{L^r} \lesssim t\os{d\prs{\frac{1}{r} - \frac{1}{p} - \frac{1}{q}}}\prs{\frac{\abs{y_1}}{t}}\os{\eta_1}\prs{\frac{\abs{y_2}}{t}}\os{\eta_2}\norm{f}_{L^p}\norm{g}_{L^q}
	\]
	for some $\eta_1, \eta_2 > 0$
	for each $(1/p, 1/q, 1/r)$ in the interior of $\mf{B}$. 
	Then for each $(1/p, 1/q, 1/r)$ with $r \geq p,q$ and $r > 1$ in the interior of $\mf{B}$ we have the sparse form bound
	\[
	\abss{\angss{\llac(f,g), h}} \lesssim \sup_{\mc{S}}\Lambda_{\mc{S}, p, q, r'}(f, g, h)
	\]
	for all $f,g,h$ bounded and compactly supported, where the supremum is taken over all sparse families.
\end{theorem}
Notably absent in Theorem \ref{thm0} are sparse bounds for $\tful$ due to lack of continuity estimates. In Section \ref{secFullMax} we indicate how to obtain sparse bounds for $\tful$, assuming continuity estimates, and outline how sparse bounds $\lful$ can be similarly obtained provided the appropriate continuity estimates are available. The corresponding theorem is:
\begin{theorem}\label{thm02}
	Let $L_t$ be a bilinear convolution with a compactly supported finite Borel measure $\mu$, and let 
	$\mf{B_{\star}}$ be the \textit{boundedness set} of all triples $(1/p, 1/q, 1/r)$ such that
	\[
	L_{\star,t} \cl L^p(\eun{d})\tm L^q(\eun{d}) \ra L^r(\eun{d}),
	\]
	where $L_{\star, t}$ is the corresponding unit-scale maximal operator. 
	Suppose that $\mf{B}_\star$ has nonempty interior and that $L_{\star, t}$ satisfies an $L^p$-continuity estimate
	\[
	\norm{L_{\star, t}((I - \tau_{y_1})f, (I - \tau_{y_2})g)}_{L^r} \lesssim t\os{d\prs{\frac{1}{r} - \frac{1}{p} - \frac{1}{q}}}\prs{\frac{\abs{y_1}}{t}}\os{\eta_1}\prs{\frac{\abs{y_2}}{t}}\os{\eta_2}\norm{f}_{L^p}\norm{g}_{L^q}
	\]
	for some $\eta_1, \eta_2 > 0$
	for each $(1/p, 1/q, 1/r)$ in the interior of $\mf{B}_\star$.
	Then for each $(1/p, 1/q, 1/r)$ with $r \geq p,q$ and $r > 1$ in the interior of $\mf{B}_\star$ we have the sparse form bound
	\[
	\abss{\angss{\lful(f,g), h}} \lesssim \sup_{\mc{S}}\Lambda_{\mc{S}, p, q, r'}(f, g, h)
	\]
	for all $f,g,h$ bounded and compactly supported, where the supremum is taken over all sparse families.
\end{theorem}

In Section \ref{secWeight} it is shown, as in \cite{RSS}, that extrapolation results can be applied to obtain a range of weighted bounds for $\llac$ and $\lful$ from the bounds in Theorems \ref{thm01} and \ref{thm02}, including bounds in the quasi-Banach range. In particular this extrapolation applies to the sparse bounds for maximal versions of the operators $T_t$ and $B_t$. 

While working on the current version of this paper the authors became aware of recent work by Borges, Foster, Ou, Pipher, and Zhou \cite{BFOPZ} which achieves a significantly expanded range of sparse bounds for the bilinear spherical maximal function. By working more directly with the operator they are able to obtain bounds in a wider range of dimensions, going all the way down to $d\geq 2$ for $\bful$ and to $d=1$ for $\blac$. Further, in an exciting development they are able to eliminate the condition $p,q \leq r$ from the range indices $(p,q,r)$ for which sparse bounds are established, which is in stark contrast to our general bounds as well as the bounds obtained in \cite{RSS}. Even more recently, after doing revisions on the paper Borges, Foster, and Ou \cite{BFO25} were able to obtain sparse domination for $\tlac$ when $d\geq 5$, which is a significant improvement on the dimensional threshold.

\section{Specific Bounds for $T_t$, $B_t$, and their Maximal Variants}\label{secKnownBounds}

Due to the restriction $r \geq p ,q$ in Theorems \ref{thm01} and \ref{thm02}, we must have
\[
\frac{1}{p} + \frac{1}{q} \geq \frac{2}{r} \quad \Rightarrow \quad \frac{r}{ \frac{pq}{p + q} } \geq 2,
\]
so we need an $L^p$-improving factor of at least 2 over the H\"older exponent $\frac{pq}{p+q}$ for a given pair of exponents $p, q > 1$. 
Fortunately such estimates are available for the triangle and bilinear spherical averaging operators and their single-scale maximal counterparts. For the triangle averaging operator, Iosevich, Palsson, and Sovine \cite{IPS} established the following result.
\begin{theorem}[\cite{IPS}, Thm.\ 1.2]
	In dimensions $d \geq 2$ the triangle averaging operator $T_t$ is of restricted strong-type $(\frac{2(d+1)}{d}, \frac{2(d+1)}{d}, d+1)$. 
\end{theorem}
\noindent We can interpolate this bound with the following bounds, also from \cite{IPS}, to get a further range of admissible exponents for which Theorem \ref{thm0} gives a sparse form bound for $\tlac$.
\begin{theorem}[\cite{IPS}, Thm.\ 1.1]
The triangle averaging operator $T_1$ satisfies the bound
\[
T_1\cl L\os{\frac{d+1}{d}}(\eun{d})\tm L\os{\frac{d+1}{d}}(\eun{d}) \ra L\os{s}(\eun{d}), \quad \text{ for all }~~~ s \in [\textstyle{\frac{d+1}{2d}}, 1] ~\text{ and }~~ d\geq 2,
\]
Moreover, 
\[
T_1\cl L\os{p}(\eun{d})\tm L\os{q}(\eun{d}) \ra L\os{1}(\eun{d})
\]
if and only if $(\frac{1}{p}, \frac{1}{q})$ lies in the convex hull of the points $\{(0,1), (1,0), (\frac{d}{d+1}, \frac{d}{d+1})\}$.
\end{theorem}
\noindent Finally, adding in the trivial bound $|T_{1}(f,g)(x)| \leq \| f\|_{L^{\infty}} A_{1}|g|(x)$ and the $L^p$ improving estimate $A_1 : L^{\frac{d+1}{d}}\rightarrow L^{d+1}$ one gets that trivially $T_{1}:L^{\infty} \times L^{\frac{d+1}{d}} \rightarrow L^{d+1}$ and symmetrically $T_{1}:L^{\frac{d+1}{d}}\times L^{\infty}  \rightarrow L^{d+1}$. Thus we have the following corollary of Theorem \ref{thm0}.
\begin{corollary}
	Let $(1/p, 1/q, 1/r)$ be in the interior of the set $S$ formed by intersection of the convex hull 
	\[
	\textstyle\conv\cbrks{\prs{0,0,0}, \prs{0, 1, 1}, \prs{ 0, \frac{d}{d+1}, \frac{1}{d+1} }, \prs{1, 0, 1}, \prs{\frac{d}{d+1}, 0, \frac{1}{d+1}}, \prs{\frac{d}{d+1}, \frac{d}{d+1}, 1}, \prs{\frac{d}{d+1}, \frac{d}{d+1}, \frac{2d}{d+1}}}
	\]
	with the half-spaces $r \geq p$ and $r \geq q$ when $d\geq 13$. Then
	\begin{equation*}
	\abs{\angss{\tlac(f,g), h}} \lesssim \sup_{\mc{S}}\Lambda_{\mc{S}, p, q, r'}(f, g, h)
	\end{equation*}
	for all  $f \in L^p$, $g \in L^q$, and $h \in L^r$,
	where the supremum is over all sparse collections.
\end{corollary}

\noindent For a wider range of estimates we direct the reader to \cite{BFO25}.

For the single-scale maximal triangle averaging operator, we have by a simple estimate using the known bounds due to Schlag (see \cite{SchlagSogge}) for the single-scale maximal spherical averaging operator $\sup_{s\in[1,2]}\abs{\mc{A}_s(f)(x)}$, that
\[
T_{\star, t} \cl L^p \tm L^\nf \ra L^q \quad \text{ and } \quad T_{\star, t} \cl L^\nf \tm L^p \ra L^q
\]
for $(1/p,1/q)$ in the convex hull of the points
\[
M = (0,0), \quad N = \prs{\frac{d-1}{d}, \frac{d-1}{d}}, \quad P = \prs{\frac{d-1}{d}, \frac{1}{d}}, \quad \text{add} \quad Q = \prs{ \frac{d^2 - d}{d^2 + 1},  \frac{d-1}{d^2 + 1}}. 
\]
These bounds include points with a maximum $L^p$-improving factor of $d$ in dimension $d$. 
These can be interpolated against the bounds for $T_{\star, t}$ that are a trivial consequence of the following bounds for the full maximal triangle averaging operator obtained by Cook, Lyall, and Magyar~\cite{CLM}, to obtain a range of $L^p$-improving bounds that can be applied along with Theorem \ref{thm02} to obtain sparse bounds for $\tful$, assuming one was able to get continuity estimates too. 
\begin{theorem}[\cite{CLM}, Thm.\ 4]
	The operator $\tful$ satisfies the bounds 
	\[
		\tful \cl L\os{l\frac{d}{d-1}}\tm L\os{l\frac{d}{d-1}} \ra L\os{\frac{l}{2}\frac{d}{d-1}} \quad \text{ with } \quad l = \frac{m}{m-1},
	\]
	for $m\geq 2$ an integer and $d \geq 2m$. 
\end{theorem}
\noindent Thus we have the following corollary of Theorem \ref{thm02}.
\begin{corollary}\label{cor1}
	Let $m\geq 2$ and the dimension $d \geq 2m$, and let $(1/p, 1/q, 1/r)$ be in the interior of the set $S'$ formed by intersection of the convex hull 
	\begin{align*}
		\conv &\textstyle\left\{\prs{0,0,0}, \prs{ \frac{d-1}{d}, 0, \frac{d-1}{d} }, 
			\prs{ 0, \frac{d-1}{d}, \frac{d-1}{d} },
			\prs{ \frac{d-1}{d}, 0, \frac{1}{d} }, 
			\prs{ 0, \frac{d-1}{d}, \frac{1}{d} }, 
			\prs{ \frac{d^2 - d}{d^2 + 1}, 0,  \frac{d-1}{d^2 + 1}},
			\right.\nlspace
			&\hspace*{20mm} \textstyle\left.
				\prs{ 0, \frac{d^2 - d}{d^2 + 1}, \frac{d-1}{d^2 + 1}},
				\prs{ \frac{d-1}{ld}, \frac{d-1}{ld}, \frac{2(d-1)}{ld} }
				\right\}
	\end{align*}
	with the half-spaces $r \geq p$ and $r \geq q$, where $l = \frac{m}{m-1}$. Further, assume a continuity estimate as is needed for Theorem \ref{thm02} is available for dimension $d$. Then
	\begin{equation*}
	\abs{\angss{\tful(f,g), h}} \lesssim \sup_{\mc{S}}\Lambda_{\mc{S}, p, q, r'}(f, g, h)
	\end{equation*}
	for all  $f \in L^p$, $g \in L^q$, and $h \in L^r$,
	where the supremum is over all sparse collections.
\end{corollary}

\begin{example}
	In the case where $d = 10$ and $m = 5$, the set $S'$ in Corollary \ref{cor1} is the convex hull
	\[
	\textstyle\conv\cbrks{ \prs{0,0,0}, 
		\prs{\frac{4}{5}, \frac{1}{10}, \frac{1}{10}}, 
		\prs{\frac{1}{10}, \frac{4}{5}, \frac{1}{10}}, 
		\prs{\frac{81}{101}, \frac{9}{101}, \frac{9}{101}},
		\prs{\frac{9}{101}, \frac{81}{101}, \frac{9}{101}}, 
		\prs{\frac{288}{535}, \frac{288}{535}, \frac{288}{535}}
	},
	\]
	and is plotted below. We emphasize that this is conditional on having the required continuity estimate and note that no such continuity estimates exist yet.
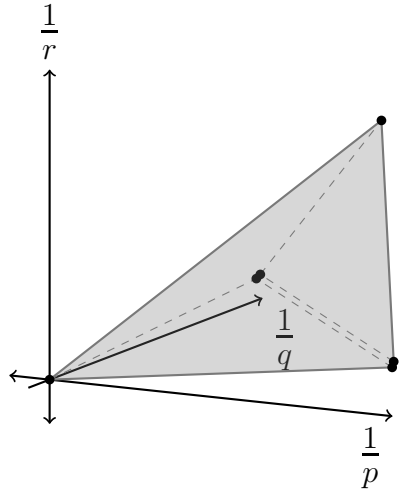
	\begin{figure}[H]
		\centering
		\begin{tikzpicture}%
		[x={(0.884382cm, -0.095077cm)},
		y={(0.466763cm, 0.180196cm)},
		z={(-0.000025cm, 0.979025cm)},
		scale=6.000000,
		back/.style={dashed, thin},
		edge/.style={color=gray!95!black, thick},
		facet/.style={fill=gray!95!black,fill opacity=0.300000},
		vertex/.style={inner sep=1pt,circle,draw=black!25!black,fill=black!75!black,thick}]
		%
		%
		
		\draw[color=black,thick,<->] (-0.1,0,0) -- (0.85,0,0) node[anchor=north east]{\LARGE$\frac{1}{p}$};
		\draw[color=black,thick,->] (0,-0.1,0) -- (0,1,0) node[anchor=north west]{\LARGE$\frac{1}{q}$};
		\draw[color=black,thick,<->] (0,0,-0.1) -- (0,0,0.7) node[anchor=south]{\LARGE$\frac{1}{r}$};
		\coordinate (0.00000, 0.00000, 0.00000) at (0.00000, 0.00000, 0.00000);
		\coordinate (0.80000, 0.10000, 0.10000) at (0.80000, 0.10000, 0.10000);
		\coordinate (0.80198, 0.08911, 0.08911) at (0.80198, 0.08911, 0.08911);
		\coordinate (0.08911, 0.80198, 0.08911) at (0.08911, 0.80198, 0.08911);
		\coordinate (0.10000, 0.80000, 0.10000) at (0.10000, 0.80000, 0.10000);
		\coordinate (0.53832, 0.53832, 0.53832) at (0.53832, 0.53832, 0.53832);
		\draw[edge,back] (0.00000, 0.00000, 0.00000) -- (0.08911, 0.80198, 0.08911);
		\draw[edge,back] (0.80000, 0.10000, 0.10000) -- (0.10000, 0.80000, 0.10000);
		\draw[edge,back] (0.80198, 0.08911, 0.08911) -- (0.08911, 0.80198, 0.08911);
		\draw[edge,back] (0.08911, 0.80198, 0.08911) -- (0.10000, 0.80000, 0.10000);
		\draw[edge,back] (0.10000, 0.80000, 0.10000) -- (0.53832, 0.53832, 0.53832);
		\node[vertex] at (0.10000, 0.80000, 0.10000)     {};
		\node[vertex] at (0.08911, 0.80198, 0.08911)     {};
		\fill[facet] (0.80000, 0.10000, 0.10000) -- (0.53832, 0.53832, 0.53832) -- (0.00000, 0.00000, 0.00000) -- (0.80198, 0.08911, 0.08911) -- cycle {};
		\draw[edge] (0.00000, 0.00000, 0.00000) -- (0.80198, 0.08911, 0.08911);
		\draw[edge] (0.00000, 0.00000, 0.00000) -- (0.53832, 0.53832, 0.53832);
		\draw[edge] (0.80000, 0.10000, 0.10000) -- (0.80198, 0.08911, 0.08911);
		\draw[edge] (0.80000, 0.10000, 0.10000) -- (0.53832, 0.53832, 0.53832);
		\node[vertex] at (0.00000, 0.00000, 0.00000)     {};
		\node[vertex] at (0.80000, 0.10000, 0.10000)     {};
		\node[vertex] at (0.80198, 0.08911, 0.08911)     {};
		\node[vertex] at (0.53832, 0.53832, 0.53832)     {};
		\end{tikzpicture}
	\end{figure}
\end{example}

Jeong and Lee \cite{JeongLee} prove the following range of $L^p$-improving bounds for the single-scale bilinear spherical averaging operator $B_{\star, t}$, which trivially imply the same bounds for $B_t$. 
\begin{theorem}[\cite{JeongLee}, Thm.\ 3.1]
	Let $d \geq 2$, $1 \leq p,q, \leq \nf$, and $0 < r \leq d$ or $\frac{d(d-1)}{d - 2} \leq r < \nf$. Then the estimate
	\[
	\norms{B_{\star, t}(f,g)}_{L^r} \lesssim t\os{d\prs{\frac{1}{r} - \frac{1}{p} - \frac{1}{q}}}\norm{f}_{
		l^p}\norm{g}_{L^q}
	\] 
	holds for $\frac{1}{r} \leq \frac{1}{p} + \frac{1}{q} < \min\cbrks{ \frac{2d-1}{d}, 1 + \frac{d}{r} }$. 
\end{theorem}
\noindent The corresponding bounds obtained for $B_t$ can be interpolated against the following $L^p$-improving bounds for $B_t$ obtained by Iosevich, Palsson, and Sovine \cite{IPS}.
\begin{theorem}[\cite{IPS}]
	In dimensions $d\geq 2$ the bilinear spherical averaging operator $B_t$ is bounded for $(1/p, 1/q, 1/r)$ in the convex hull of the set
	\[
		\cbrks{(1,0,1), (0,1,1), (1,1,1), (1,1, 2), (0,0,0) }. 
	\]
\end{theorem}
\noindent The bounds obtained from the results mentioned here and interpolation can be applied with Theorem \ref{thm0} to obtain sparse bounds for $\blac$ and $\bful$. The following example shows a concrete range in one case.

\begin{example}
	In the case where $d = 10$, Theorem \ref{thm0} gives sparse bounds for $\bful$ when $(1/p,1/q,1/r)$ lies in the interior of the set 
	\begin{align*}
		&\textstyle\conv\left\{
			\prs{1, \frac{9}{10}, \frac{9}{10}},
			\prs{\frac{9}{10}, 1, \frac{9}{10}},
			\prs{\frac{9}{10}, 1, \frac{1}{10}},
			\prs{1, \frac{9}{10}, \frac{1}{10}},\right.\nlspace
		&\textstyle\hspace*{3cm}\left.
			\prs{1, \frac{1}{10}, \frac{1}{10}},
			\prs{\frac{1}{10}, 1, \frac{1}{10}},
			\prs{\frac{1}{10}, \frac{1}{10}, \frac{1}{10}},
			\prs{\frac{19}{20},\frac{19}{20},\frac{19}{20}}
		\right\}\nlspace
	\textstyle\cup&\textstyle~\conv\left\{
		\prs{0,0,0},
		\prs{0,1,0},
		\prs{1,0,0},
		\prs{\frac{8}{9}, 1, \frac{4}{45}},
		\prs{1, \frac{8}{9}, \frac{4}{45}},\right.\nlspace
	&\textstyle\hspace*{3.5cm}\left.
		\prs{1, \frac{4}{45}, \frac{4}{45}},
		\prs{\frac{4}{45}, 1, \frac{4}{45}},
		\prs{\frac{4}{45}, \frac{4}{45}, \frac{4}{45}}
	\right\}.
	\end{align*}
\end{example}
\noindent For a wider range of estimates we direct the reader to \cite{BFOPZ}. 

\section{Abstract Sparse Domination Theorem}\label{secAbsDom}

Here we will prove a sparse domination theorem using an abstract version of the procedure presented in \cite{RSS}.
The proofs in this section follow those in \cite{RSS} closely, with appropriate modifications for application to a general class of operators. This theorem and its proof show how continuity and $L^p$-improving properties of bilinear convolutions\footnote{Below we state the continuity estimate in simultaneous form but by interpolation this is equivalent to continuity in each variable separately, since we just require some positive exponents $\eta_i$. We also state this estimate in the rescaled version, but this version follows immediately from the same estimate with $t = 1$ and a change of variables.}
\[
(f,g) \mapsto L_t(f,g)(x) := [(f\otimes g)\star \mu_t](x,x)
\]
with compactly supported finite Borel measures give rise to sparse bounds for the corresponding lacunary maximal operators
\[
\llac(f,g)(x) := \sup_{j \in \znm}\abss{L_{2^j}(f,g)(x)}. 
\]
In Section \ref{secFullMax} we indicate how the proofs here can be modified to prove sparse bounds for the full maximal triangle averaging operator, and indicate how similar techniques will work for convolutions with other measures as long as one has the required continuity properties for the corresponding single-scale maximal operator. We now remind the reader of the statement of Theorem \ref{thm01}.
\begin{theorem}\label{thm1}
	Let $L_t$ be a bilinear convolution with a compactly supported finite Borel measure $\mu$ satisfying an $L^p$-improving estimate 
	\[
	L_t \cl L^p \tm L^q \ra L^r \quad \text{for some $(p,q,r)$ with } \quad \frac{1}{p} + \frac{1}{q} > \frac{1}{r}, 
	\]
	and let 
	$\mf{B}$ be the \textit{boundedness set} of all triples $(1/p, 1/q, 1/r)$ such that 
	\[
	L_t \cl L^p(\eun{d})\tm L^q(\eun{d}) \ra L^r(\eun{d}). 
	\]
	Suppose that $L_t$ satisfies an $L^p$-continuity estimate
	\[
	\norm{L_t((I - \tau_{y_1})f, (I - \tau_{y_2})g)}_{L^r} \lesssim t\os{d\prs{\frac{1}{r} - \frac{1}{p} - \frac{1}{q}}}\prs{\frac{\abs{y_1}}{t}}\os{\eta_1}\prs{\frac{\abs{y_2}}{t}}\os{\eta_2}\norm{f}_{L^p}\norm{g}_{L^q}
	\]
	for some $\eta_1, \eta_2 > 0$
	for each $(1/p, 1/q, 1/r)$ in the interior of $\mf{B}$. 
	Then for each $(1/p, 1/q, 1/r)$ with $r \geq p,q$ and $r > 1$ in the interior of $\mf{B}$ we have the sparse form bound
	\[
	 \abss{\angss{\llac(f,g), h}} \lesssim \sup_{\mc{S}}\Lambda_{\mc{S}, p, q, r'}(f, g, h)
	\]
	for all $f, g, h$ bounded and compactly supported, where the supremum is taken over all sparse families.
\end{theorem}
%

Theorem \ref{thm1} will follow from Lemma \ref{lem1}, which is a version of Theorem \ref{thm1} for indicator functions. We postpone the proof of Theorem \ref{thm1} until after the proof of Lemma \ref{lem2} for the sake of efficiency, since we will repeat some of the arguments from the proof of Lemma \ref{lem1} in the proof of Theorem \ref{thm1}.

\begin{lemma}\label{lem1} Let $L_t$ be as in the statement of Theorem \ref{thm1}, and
	let $(1/p, 1/q, 1/r)$ with $r \geq p,q$ and $r > 1$ be in the boundedness set $\mf{B}$ of $L_1$. Then for compactly supported measurable indicator functions $f = 1_{F}$ and $g = 1_{G}$ and a compactly supported bounded function $h$ there exists a sparse collection $\mc{S}$ such that
	\begin{equation*}
		\angss{\llac(f,g), h} \lesssim \Lambda_{\mc{S}, p, q, r'}(f, g, h).
	\end{equation*}
\end{lemma}

\noindent We will prove Lemma \ref{lem1} using the following lemma.

\begin{lemma}\label{lem2}
	Let $L_t$ be as in the statement of Theorem \ref{thm1}, and
	let $(1/p, 1/q, 1/r)$ with $r \geq p,q$ and $r > 1$ be in the boundedness set $\mf{B}$ of $L_1$ and let $f = 1_{F}$ and $g = 1_{G}$ be measurable indicator functions supported in a dyadic cube $Q_0$ and $h$ a bounded function also supported in $Q_0$. Let $C_0 > 1$ and $\mc{D}_0$ a collection of dyadic subcubes of $Q_0$ such that 
	\[
	\sup_{Q' \in \mc{D}_0} \sup_{Q \colon Q' \se Q \se Q_0}
	\prs{ \frac{\angs{f}_{Q, p}}{\angs{f}_{Q_0, p}} + 
		\frac{\angs{g}_{Q, q}}{\angs{g}_{Q_0, q}} + 
		\frac{\angs{h}_{Q, r'}}{\angs{h}_{Q_0, r'}} } 
	\leq C_0.
	\]
	Then, if $\{B_Q\}$ is a family of pairwise disjoint sets indexed by $Q \in \mc{D}_0$ with $B_Q \se Q$ and $h_Q := h1_{B_Q}$, we have the estimate
	\[
	\sum_{Q \in \mc{D}_0}{\angss{L_{Q}^j(f,g), h_Q}} \lesssim |Q_0|\angs{f}_{Q_0, p}\angs{g}_{Q_0, q}\angs{h}_{Q_0, r'},
	\]
	where $L_{Q}^j$ is as defined in the proof of Lemma \ref{lem1}.
\end{lemma}

\subsection{Proof of Lemma \ref{lem1} using Lemma \ref{lem2}.} \label{subsecLem1}
We will first reduce the bounds for $\llac$ to bounds for a family of dyadic maximal operators. If $\mc{D}'$ is a dyadic lattice of cubes of sidelengths $\frac{1}{3}2^j, j \in \znm$, then we can partition the set of triples of cubes in $\mc{D}'$ into the disjoint union
\[
\{3Q \cl Q \in \mc{D}'\} = \bigcup_{i=1}^{3^d}\mc{D}^i,
\]
where each $\mc{D}^i$ is a lattice of cubes of dyadic sidelengths. Then we can write (with $\mc{D}'_q$ the cubes in $\mc{D}'$ with sidelength $2^q$ times the fundamental sidelength), 
\begin{align*}
L_{2\os{q-3}}(f,g)(x) &= L_{2\os{q-3}}\prsm{\sum_{Q \in \mc{D}_q'}f1_{Q},\, \sum_{Q' \in \mc{D}_q'}g1_{Q'}}(x)\nlspace
&= \sum_{Q, Q' \in \mc{D}_q'}L_{2\os{q-3}}(f1_Q,g1_{Q'})(x)\nlspace
&=:\sum_{Q, Q' \in \mc{D}_q'} L_{Q,Q'}(f,g)(x). 
\end{align*}
By rescaling the measure if necessary, we can assume WLOG that $\diam(\sppt\mu) \leq 1/2$, so that $\sup\{\abs{x - y} : (x,y) \in \sppt\mu\}\leq 1$ also. 
Then if $L_{Q,Q'}$ does not vanish, then there are points $p \in Q$ and $p' \in Q'$ with $|p - p'| \leq 2\os{q-3}$. This implies that $Q' \se 3Q$. Let $Q(1), \ldots, Q(3^d)$ be an enumeration of the cubes of sidelength $l_Q$ that make up $3Q$. Then we can write
\[
\sum_{Q, Q' \in \mc{D}_q'}L_{Q,Q'}(f,g)(x) = \sum_{Q \in \mc{D}_q'}\sum_{j=1}^{3^d}L_{Q, Q(j)}(f,g)(x),
\]
and observe that $L_{Q,Q(j)}$ is supported in $3Q$. Hence we define for a cube $Q \in \mc{D}^i$, and for $i,j \in \{1, \ldots, 3^d\}$,
\[
L^j_{Q}(f,g)(x):= L_{\frac{1}{3}Q, (\frac{1}{3}Q)(j)}(f,g)(x), \quad \text{with} \quad \sppt(L_Q^j)\se Q.
\]
Then we have for the lacunary maximal operator
\begin{align*}
\sup_{q \in \znm} \abss{L_{2\os{q-3}}(f,g)(x)} &= \sup_{q \in \znm} \absm{  
	\sum_{i,j=1}\os{3^d}\sum_{Q \in \mc{D}^i_q} L^j_{Q}(f,g)(x)
}\nlspace
&\leq \sum_{i,j=1}\sup_{Q \in \mc{D}^i}\abss{L^j_{Q}(f,g)(x)  } =: \sum_{i,j=1}\mc{M}_{i,j}(f,g)(x). 
\end{align*}
Thus we have reduced the bounds for $\angss{\llac(f,g), h}$ to bounds for the form 
\[
\angss{\mc{M}_{j}(f,g), h}, \quad \text{where} \quad \mc{M}_j(f,g)(x):= \sup_{Q \in \mc{D}}\abs{L^j_Q(f,g)(x)},
\]
where $\mc{D}$ is any dyadic lattice.
This parallels the analyses in \cite{LaceySpherical} and \cite{RSS}. 

We assume (WLOG) that $f,g,h$ are supported in a cube $\widetilde{Q}$. 
Now by support considerations and the definition of the operator $L_Q^j$ it follows that there is some large ancestor cube $Q_0$ of $\widetilde{Q}$ such that, since $\sppt h \se \widetilde{Q}$, 
\[
\angss{L^j_{Q'}(f,g), h} = 0 \quad \text{for all $Q'$ properly containing $Q_0$.}
\]
Hence in our analysis of $\mc{M}_{j}$ we can restrict our attention to the supremum of $|L^j_Q|$ for $Q$ a dyadic subcube of $Q_0$. This reduction is crucial because it allows us to construct a sparse family inductively, starting with $Q_0$ and moving into appropriate subcubes.

Now we obtain the desired sparse family by applying the following Lemma \ref{lem6} with $\mc{F} = \mc{D}\cap Q_0$ which looks at only cubes in $\mc{D}$ that are subcubes of $Q_0$. This is the first place in the proof where we used that $f,g$ are indicator functions, and in the proof of Lemma \ref{lem6} this assumption is only used when applying Lemma \ref{lem2}. This concludes the proof of Lemma \ref{lem1}.\qed 

\begin{lemma}\label{lem6}
	Let $L_t$, $f,g,h$ and $p,q,r$ be as in the statement of Lemma \ref{lem1}. 
	Let $\mc{D}$ be a dyadic lattice containing $Q_0$ and let $\mc{F}$ be a subcollection of cubes in $\mc{D}\cap Q_0$ that contains $Q_0$. Then there is a sparse subcollection $\mc{S}$ of $\mc{F}$ such that
	\[
	\angss{\sup_{Q \in \mc{F}}L^j_Q(f,g), h} \lesssim \sum_{Q \in \mc{S}}|Q|\angs{f}_{Q, p}\angs{g}_{Q, q}\angs{h}_{Q, r'}.
	\]
\end{lemma}
\noindent\textit{Proof of Lemma \ref{lem6}.} 
Choose a $C_0 > 1$ and, as in \cite{LaceySpherical} and \cite{RSS}, define
\[
\mc{E}_{Q_0}:= \cbrks{ \text{maximal $Q\in \mc{F}$ such that }
	\max\cbrks{ \frac{\angs{f}_{Q, p}}{\angs{f}_{Q_0, p}}, \frac{\angs{g}_{Q, q}}{\angs{g}_{Q_0, q}}, \frac{\angs{h}_{Q, r'}}{\angs{h}_{Q_0, r'}} } > C_0},
\]
and let $E_{Q_0} := \bigcup_{P \in \mc{E}_{Q_0}}P$ and $F_{Q_0} := Q_0 \sm E_{Q_0}$.
Let $E_f$ be the disjoint union of cubes in $\mc{E}_{Q_0}$  for which $\angs{f}_{Q, p} > C_0 \angs{f}_{Q_0, p}$ holds. Then we have by writing as a weighted average,
\begin{align*}
\frac{1}{|Q_0|}\int_{Q_0}\abs{f}^{p}\dx{x}  &= \prs{\frac{|E_f|}{|Q_0|}\cdot \frac{1}{|E_f|}\int_{E_f} + \frac{|Q_0\sm E_f|}{|Q_0|}\cdot \frac{1}{|Q_0 \sm E_F|}\int_{Q_0 \sm E_f}}\abs{f}^{p}\dx{x}\nlspace
&\geq \frac{|E_f|}{|Q_0|}\cdot \frac{1}{|E_f|}\int_{E_f} \abs{f}\os{p}\dx{x}\nlspace
&= \frac{|E_f|}{|Q_0|} \prsm{\sum_{\substack{P \in \mc{E},\\P \se E_f}} \frac{|P|}{|E_f|}\cdot \frac{1}{|P|}\int_P \abs{f}\os{p}\dx{x} }\nlspace
&> \frac{|E_f|}{|Q_0|} \cdot C_0\os{p}\cdot \frac{1}{|Q_0|}\int_{Q_0}\abs{f}^{p}\dx{x},
\end{align*}
implying that
\[
|E_f| < \frac{|Q_0|}{C_0\os{p}}. 
\]
So in particular we can choose $C_0$ large enough that $|E_f| < \frac{1}{6}|Q_0|$. Since similar estimates hold for analogous sets $E_g, E_h$, we can choose $C_0$ large enough so that
\[
|E_{Q_0}| \leq \frac{1}{2}|Q_0|. 
\]

Now we define 
\[
\mc{D}_0 := \cbrks{ Q \in \mc{F} ~\cl~ \text{$Q$ is not a subset of a cube in $\mc{E}_{Q_0}$}},
\]
and note in particular that $Q_0 \in \mc{D}_0$. 
If $Q$ is contained in a cube $Q' \in \mc{F}$ such that it holds
\begin{equation}\label{eq1}
\max\cbrks{ \frac{\angs{f}_{Q', p}}{\angs{f}_{Q_0, p}}, \frac{\angs{g}_{Q', q}}{\angs{g}_{Q_0, q}}, \frac{\angs{h}_{Q', r'}}{\angs{h}_{Q_0, r'}} } > C_0,
\end{equation}
then $Q$ is contained in a maximal subcube $Q' \in \mc{F}$ for which \eqref{eq1} holds, a contradiction. Hence 
\[
Q \in \mc{D}_0 \quad \Longrightarrow 
\quad \sup_{Q'\in \mc{F} \cl Q \se Q' \se Q_0}
\max\cbrks{ \frac{\angs{f}_{Q, p}}{\angs{f}_{Q_0, p}}, \frac{\angs{g}_{Q, q}}{\angs{g}_{Q_0, q}}, \frac{\angs{h}_{Q, r'}}{\angs{h}_{Q_0, r'}} } \leq C_0. 
\]

Now we will linearize the supremum in the bilinear form as in \cite{RSS}. We define for each $Q \in \mc{D}_0$,
\[
H_{Q}:= \cbrks{x \in Q\, \cl L_{Q}^j(f,g)(x) \geq \frac{1}{2}\sup_{Q \in \mc{D}_0}|L^j_Q(f,g)(x)|}, \quad \text{and} \quad B_{Q}:= H_Q\sm \bigcup_{P \subsetneq Q}H_P,
\]
so that the $B_Q$'s are pairwise disjoint and $\bigcup_{Q \in \mc{D}_0}B_Q = Q_0$. Then (recalling $f,g,h\geq 0$)
\begin{align*}
{\angss{ \sup_{Q \in \mc{D}_0}|L^j_Q(f,g)(x)|, h }} &= \sum_{Q \in \mc{D}_0}\int_{B_{Q}}\sup_{Q' \in \mc{D}_0}|L^j_{Q'}(f,g)(x)| h(x)\dx{x}\nlspace
&\leq 2\sum_{Q \in \mc{D}_0}\int_{B_{Q}}L_Q^j(f,g)(x) h(x)\dx{x}\nlspace
&= 2\sum_{Q \in \mc{D}_0}\angss{L_Q^j(f,g), h_Q} {\leq 2{\angss{ \sup_{Q \in \mc{D}_0}|L^j_Q(f,g)(x)|, h }} },
\end{align*}
where $h_Q := h1_{B_Q}$, so we can equivalently estimate either the form or its linearization.  

Now we have
\begin{align*}
{\angss{\sup_{Q \in \mc{F}}|L^j_Q(f,g)|, h}} &\leq \angss{\sup_{Q \in \mc{D}_0}|L^j_Q(f,g)| , h} + \sum_{P \in \mc{E}_{Q_0}}\angss{\sup_{Q \in \mc{F}\cap P}|L^j_Q(f,g)|, h1_{P}} \nlspace
&\approx \sum_{Q \in \mc{D}_0}\angss{L_Q^j(f,g), h_Q}
+ \sum_{P \in \mc{E}_{Q_0}}\angss{\sup_{Q \in \mc{F}\cap P}|L^j_Q(f,g)|, h1_{P}}.
\end{align*}
Lemma \ref{lem2} gives for the first term 
\[
\sum_{Q \in \mc{D}_0}\angss{L_Q^j(f,g), h_Q} \lesssim |Q_0| \angs{f}_{Q_0, p}\angs{g}_{Q_0, q}\angss{h}_{Q_0, r'},
\]
and we can apply the same argument we've just used to each term of the sum in the second term. Since the cubes $P \in \mc{E}_{Q_0}$ are disjoint with total measure $< \frac{1}{2}|Q_0|$, this recursive procedure suffices to construct a sparse family $\mc{S}$. 
More specifically, we add $Q_0$ to $\mc{S}$, then repeat the procedure starting with each $P\in \mc{E}_{Q_0}$ in place of $Q_0$. When we enter the part of the construction starting with $P \in \mc{E}_{Q_0}$, we can WLOG multiply $f, g$, and $h$ by $1_P$, so that each function is supported in $P$. 
By construction $\mc{S}$ will be a subfamily of $\mc{F}$.  This concludes the proof of Lemma \ref{lem6}.
\qed


\subsection{Proof of Theorem \ref{thm1} using Lemma \ref{lem1}.}
We proceed as in \cite{RSS}. 
We can assume WLOG that $f, g, h$ are nonnegative. 
This proof is essentially a modification of the proofs of Lemmas \ref{lem1} and \ref{lem6}. 
Specifically, we will proceed as in the proof of Lemma \ref{lem1} and will prove a version of Lemma \ref{lem6} that holds when $f,g$ are bounded, compactly supported functions: We will show that with these modified hypotheses of Lemma \ref{lem6}, if $\mc{H}$ is a collection of dyadic subcubes of a fixed dyadic cube $Q_0$ and $\mc{H}$ contains $Q_0$, then there is a sparse subcollection $\mc{S}$ of $\mc{H}$ such that
\[
\angss{\sup_{Q \in \mc{H}}L^j_Q(f,g), h} \lesssim \sum_{Q \in \mc{S}}|Q|\angs{f}_{Q, p}\angs{g}_{Q, q}\angs{h}_{Q, r'}.
\]

We first assume that $f$ is a bounded compactly supported function and $g=1_G$ is a compactly supported characteristic function with $G$ measurable, and begin by constructing the families of cubes $\mc{E}_{Q_0}\se \mc{H}$ and $\mc{D}_0\se\mc{H}$ 
as in the proof of  Lemma \ref{lem6}, starting with $\mc{F} = \mc{H}$. 
Now by the recursive argument in the proof of Lemma \ref{lem6} it suffices to show that
\[
\angss{\sup_{Q \in \mc{D}_0}L^j_Q(f,g), h} 
\lesssim
	|Q_0| \angs{f}_{Q_0, p}\angs{g}_{Q_0, q}\angss{h}_{Q_0, r'}.  
\]
(In the proof of Lemma \ref{lem6} we got this by applying Lemma \ref{lem2}.) As in \cite{RSS}, in this case we start by approximating the function $f$ by
\[
f \approx \sum_{m=-\nf}^{\nf} 2\os{m+1} f_m, \quad \text{where} \quad f_m = 1_{E_m}, \quad \text{with} \quad E_m:= \cbrks{ x \in Q_0 \cl 2^m \leq f < 2\os{m+1} }.
\]
Then we have by Lemma \ref{lem6} with $\mc{F}$ equal to the set $\mc{D}_0$ constructed here and with a triple of indices $(\ol{p}, q, r')$ with $\ol{p} < p$ (we can do this because we are using the $L^p$-improving property of the operator and we consider indices in the interior of the boundedness region),
\begin{align*}
	\angss{\sup_{Q \in \mc{D}_0}L^j_Q(f,g), h} &\approx \sum_{m=-\nf}^{\nf} 2\os{m+1}\angss{\sup_{Q \in \mc{D}_0}L^j_Q(f_m,g), h}\nlspace
	&\leq \sum_{m=-\nf}^{\nf} 2\os{m+1} 
	\sum_{Q \in \mc{S}_m}
	|Q|\angs{f_m}_{Q, \ol{p}}\angs{g}_{Q,{q}}\angss{h}_{Q, r'},
\end{align*}
where the $\mc{S}_m$ are sparse subfamilies of $\mc{D}_0$. 
Then by the stopping-time condition in the definition of $\mc{D}_0$ the last line above is (noting that $C_0$ in the proof of Lemma 1 can be chosen depending only on $p,q,r$)
\[
\lesssim \angs{g}_{Q_0, q}\angss{h}_{Q_0, r'}
\sum_{m=-\nf}^{\nf} 2\os{m+1} 
\sum_{Q \in \mc{S}_m}
|Q|\angs{f_m}_{Q, \ol{p}}. 
\]
By Lemmas \ref{lem3}, \ref{lem4}, and \ref{lem5} this is 
\[
\lesssim |Q_0|\angs{f}_{Q_0, p}\angs{g}_{Q_0, q}\angss{h}_{Q_0, r'},
\]
as required. This allows us to construct the desired sparse family, establishing a version of Lemma \ref{lem6} in which one of $f, g$ is an indicator function and the other is an abitrary bounded, compactly supported, non-negative function. 

Now we will assume that $f,g,h$ are general bounded, compactly supported, non-negative functions, and we will use the version of Lemma \ref{lem6} just established to obtain a sparse family for these functions. We again start by constructing the families of cubes $\mc{E}_{Q_0}$ and $\mc{D}_0$
as in the proof of  Lemma \ref{lem6}, this time applied with $\mc{F} = \mc{D}\cap Q_0$,
and again it suffices to show that
\[
\angss{\sup_{Q \in \mc{D}_0}L^j_Q(f,g), h} 
\lesssim
|Q_0| \angs{f}_{Q_0, p}\angs{g}_{Q_0, q}\angss{h}_{Q_0, r'}.  
\]
We approximate the function $g$ as we approximated $f$ in the previous case, and we get
\begin{align*}
		\angss{\sup_{Q \in \mc{D}_0}L^j_Q(f,g), h} &\approx \sum_{m=-\nf}^{\nf} 2\os{m+1}\angss{\sup_{Q \in \mc{D}_0}T^j_Q(f,g_m), h}\nlspace
	&\leq \sum_{m=-\nf}^{\nf} 2\os{m+1} 
	\sum_{Q \in \mc{S}_m}
	|Q|\angs{f}_{Q, \ol{p}}\angs{g_m}_{Q,{q}}\angss{h}_{Q, r'},
\end{align*}
where this time we applied the result we just established with $\mc{H} = \mc{D}_0$ to construct the sparse subfamilies $\mc{S}_m$ of cubes in $\mc{D}_0$. Now we can bound the double sum in the last line just as we did in the previous case, using the stopping condition on $\mc{D}_0$ and Lemmas \ref{lem3}, \ref{lem4}, and \ref{lem5}. This establishes the desired version of Lemma \ref{lem6} and allows us to proceed with the recursion to construct the desired sparse family.
\qed

\begin{lemma}[\cite{BagchiEA}, Lemma 4.7]\label{lem3}
	On a probability space $(X, \mu)$ we have the continuous inclusion $L^p(X, d\mu) \hookrightarrow L\os{r,1}(X, d\mu)$ for $p > r$. 
\end{lemma}

\begin{lemma}[\cite{BagchiEA}, Lemma 4.8]\label{lem4}
	Let $f$ be a measurable function and define $E_m := \{ x \in Q_0 \cl 2^m \leq |f(x)|< 2\os{m+1}\}$. Then with $\mu$ the probability measure $|Q_0|\fv\dx{x}$ on $Q_0$ we have
	\[
	\sum_{m=-\nf}^\nf 2\os{m}\angs{1_{E_m}}_{Q_0, r} \lesssim \norm{f}_{L\os{r,1}(Q_0, d\mu)}. 
	\]
\end{lemma}

\begin{lemma}[\cite{BagchiEA}, Lemma 4.9]\label{lem5}
	Let $\mc{S}$ be a sparse collection of dyadic subcubes of a fixed dyadic cube $Q_0$ and let $1 \leq s < t < \nf$. Then for a bounded function $\phi$,
	\[
	\sum_{Q \in \mc{S}}|Q|\angs{\phi}_{Q,s} \lesssim |Q_0|\angss{\phi}_{Q_0, t}.
	\]
\end{lemma}

\noindent\textit{Proof.} We have by the sparsity of the collection $\mc{S}$, with $E_Q$ the sparse subset of $Q \in \mc{S}$,
\begin{align*}
	\sum_{Q \in \mc{S}}|Q|\angs{\phi}_{Q,s} & \leq \prsm{ \sum_{Q \in \mc{S}}|Q|\angs{\phi}^t_{Q,s} }\os{\frac{1}{t}}\prsm{\sum_{Q \in \mc{S}}|Q|}\os{\frac{1}{t'}}\nlspace
		&\lesssim |Q_0|\os{\frac{1}{t'}}\prsm{ \sum_{Q \in \mc{S}}|Q|\angs{|\phi 1_{Q_0}|^s}\os{\frac{t}{s}}_{Q} }\os{\frac{1}{t}}\nlspace
		&\lesssim |Q_0|\os{\frac{1}{t'}}\prsm{
			\sum_{Q \in \mc{S}} \int_{E_Q}\prs{M\os{\frac{1}{s}}(|\phi 1_{Q_0}|^s)(x)}\os{t}\dx{x}
		}\os{\frac{1}{t}}\nlspace
		&\lesssim |Q_0|\angs{\phi}_{Q_0, t},
\end{align*}
by the boundedness on $L\os{\frac{t}{s}}$ of the Hardy-Littlewood maximal operator. 
\qed

\subsection{Proof of Lemma \ref{lem2}}	\label{subsecLem2}
We will show how the proof of the corresponding lemma in \cite{RSS} can be modified in order to apply to the triangle average. 
%
To begin we recall that
\[
L_Q^j(f,g)(x) = L_{2\os{\log_2(l(Q))-3}}(1_{\frac{1}{3}Q}f,\, 1_{(\frac{1}{3}Q)(j)}g)(x). 
\]
Since we will be using cancellation properties of functions in a Calder\'on-Zygmund decomposition we need to enlarge the cutoffs inside the operator, following \cite{RSS}, to ensure that cancellation is preserved for all but finitely many scales. We define $\widetilde{Q}(j)$ to be a union of dyadic subcubes of $Q$ of sidelength $l(Q)/2$ that covers $((1/3)Q)(j)$ (which is one of the cubes obtained by dividing $Q$ into $3^d$ congruent subcubes). Then, since we consider $f, g \geq 0$, we define
\[
S_{Q,j}(f,g) := L_{2\os{\log_2(l(Q))-3}}(1_{\frac{1}{2}Q}f,\, 1_{\widetilde{Q}(j)}g)(x) \geq L_Q^j(f,g)(x). 
\]

Now we will apply a Calder\'on-Zygmund type decomposition to each of the functions $f, g$.  With the same notation as \cite{RSS} we define
\[
B_{f}:= \cbrks{ \text{maximal dyadic } P \se Q_0 \text{ such that } \frac{\angs{f}_{P, p}}{\angs{f}_{Q_0, p}} > 2C_0 },
\]
and we define $B_{g}$ analogously. Note that by the definition of the collection $\mc{D}_0$ in the statement of the lemma, if $Q \in \mc{D}_0$, $P\in B_{f}$ and $P \cap Q \not= \emptyset$, then $P$ is a proper subcube of $Q$. Now we define the good and bad parts of the decomposition
\[
f =:  \gamma_f + \beta_f,
\]
where
\begin{align*}
	\beta_f := \sum_{P \in B_f}\prss{f - \angs{f}_P}1_{P} =: \sum_{k = -\nf}^{\log_2(l(Q_0)) - 1}\beta_{f, k},
\end{align*}
where $\beta_{f,k}$ is the part of $\beta_f$ supported on the set $B_{f,k}$ of cubes in $B_f$ with sidelength $2^k$. 
We make a parallel decomposition for the function $g$.
We note that $\norm{\gamma_f}_{L^\nf} \lesssim \angs{f}_{Q_0, p}$ and $\norm{\gamma_g}_{L^\nf} \lesssim \angs{g}_{Q_0, q}$ by standard arguments, and also that, with $m$ fixed, 
\[
\angs{\beta_{f, m - k}}_{Q, p} \lesssim \angs{1_{F_{m,k}}}_{Q, p} + \angs{f}_{Q_0, p}\angs{1_{E_{m, k}}}_{Q, p},
\]
where $F_{m, k}$ are subsets of $F$ that are disjoint as $k$ varies and $E_{m, k}$ are subsets of $Q_0$ that are disjoint as $k$ varies. We can exhibit a similar bound for $\beta_{g,m-j}$. 

Now we have
\begin{align*}
	\sum_{Q \in \mc{D}_0}\angss{L^j_Q(f,g), h_Q} &\leq 
	\abss{ \sum_{Q \in \mc{D}_0}\angss{L^j_Q(\gamma_f,\gamma_g), h_Q} } +
	\abss{ \sum_{Q \in \mc{D}_0}\angss{L^j_Q(\gamma_f,\beta_g), h_Q} }  \nlspace
	&\qquad + \abss{ \sum_{Q \in \mc{D}_0}\angss{L^j_Q(\beta_f,\gamma_g), h_Q} } + 
	\abss{ \sum_{Q \in \mc{D}_0}\angss{L^j_Q(\beta_f,\beta_g), h_Q} },\nlspace
	&=: GG + GB + BG + BB,
\end{align*}
using the same notation as \cite{RSS} for the four parts. \\

\textit{\textbf{Estimate for $GG$:}} This estimate follows straightforwardly as in \cite{RSS}:
\begin{align*}
	GG &\leq \sum_{Q \in \mc{D}_0}\norm{\gamma_f}_{L^\nf}\norm{\gamma_g}_{L^\nf}\norm{h_Q}_{L^1} \nlspace
	&\lesssim \angs{f}_{Q_0, p}\angs{g}_{Q_0, q} \sum_{Q \in \mc{D}_0} \int\abs{h(x)1_{B_Q}(x)}\dx{x}\nlspace
	&= |Q_0|\angs{f}_{Q_0, p}\angs{g}_{Q_0, q}\angs{h}_{Q_0, 1}\nlspace
	&\leq |Q_0|\angs{f}_{Q_0, p}\angs{g}_{Q_0, q}\angs{h}_{Q_0, r'}. 
\end{align*}

\textit{\textbf{Estimate for $BG$ ($GB$ is similar):}} By the support considerations described above we can write, with $m = \log_2 l(Q)$, 
\begin{align*}
	BG &\leq \sum_{Q \in \mc{D}_0}\sum_{k = 1}^3\abs{\angs{ S_{Q,j}(\beta_{f, m - k},\gamma_g), h_Q}} + 
		\sum_{Q \in \mc{D}_0}\sum_{k = 4}^\nf \abs{\angs{ S_{Q,j}(\beta_{f, m - k},\gamma_g), h_Q}}\nlspace
	& =: I_{BG} + II_{BG}.
\end{align*}
The first term in this sum can be estimated using the scaled $L^p$-improving bounds for the operator $T_{2\os{m-3}}$:
\begin{align*}
	I_{BG} &\lesssim \sum_{Q \in \mc{D}_0}\sum_{k = 1}^3 2\os{dm(\frac{1}{r} - \frac{1}{p} - \frac{1}{q})}
	\norms{\beta_{f, m - k}1_{\frac{1}{2}Q}}_{L^p}\norms{\gamma_g1_{\widetilde{Q}(j)}}_{L^q}\norm{h_Q}_{L^{r'}}\nlspace
	&\lesssim \sum_{k = 1}^3 2\os{-\eta k}\sum_{Q \in \mc{D}_0}|Q| \angs{\beta_{f, m - k}}_{Q, p}\angs{\gamma_g}_{Q, q}\angs{h}_{Q, r'}\nlspace
	&\lesssim \angs{g}_{Q_0, q}\sum_{k = 1}^3 2\os{-\eta k}\sum_{Q \in \mc{D}_0}|Q| \angs{\beta_{f, m - k}}_{Q, p}\angs{h}_{Q, r'}.
\end{align*} 

It remains to get a similar estimate for $II_{BG}$. We will use a duality argument that is essentially a more abstract version of the argument in \cite{RSS}.
First we define the formal adjoints of $T_{2\os{m-3}}$, which we abbreviate here to just $S$. As usual, for nice functions we define via Fubini's theorem and a change of variable:
\begin{align*}
&~\angss{S(f,g), h} =: \angs{f, S\os{*,1}(g,h)}\nlspace
&:= \int f(x)\cdot \prs{ \int g(x - 2\os{m-3}(z - y))h(x + 2\os{m-3}y) \dx{\mu(y,z)}}\dx{x},
\end{align*}
and
\begin{align*}
&~\angss{S(f,g), h} =: \angs{g, S\os{*,2}(f,h)}\nlspace
&:= \int g(x)\cdot \prs{ \int f(x - 2\os{m-3}(y - z))h(x + 2\os{m-3}z) \dx{\mu(y,z)}}\dx{x}.
\end{align*}
For $q$ and $k$ fixed we can further write
\begin{align*}
\abss{ \angss{S_{Q,j}(\beta_{f,m- k}, \gamma_g), h_Q} } &\leq \sum_{P \in B_{f,m-k}}\abss{\angss{S_{Q,j}(1_P\beta_{f, m - k},\, \gamma_g), h_Q}}\nlspace
	&= \sum_{P \in B_{f,m-k}}\abss{\angss{1_P\beta_{f, m - k}1_{\frac{1}{2}Q},\, S^{*, 1}( \gamma_g1_{\tilde{Q}(j)},\, h_Q} )}. 
\end{align*}
Now for $k\geq 4$ any cube $P \in B_{f, m -k}$ that intersects $\frac{1}{2}Q$ or $\tilde{Q}(j)$ is strictly contained therein, so we can use the cancellation of the pieces of the bad function $\beta_{f}$ to get
\begin{align*}
&\sum_{P \in B_{f,m-k}}\abs{\int_P\beta_{f, m-k}(x)1_{\frac{1}{2}Q}(x)\cdot S^{*,1}\prss{\gamma_g1_{\widetilde{Q}(j)},\, h_Q}(x)\dx{x}}\nlspace
&= \sum_{P \in B_{f,m-k}}\frac{1}{|P|}\absm{\int_P\int_P
	\underbrace{\beta_{f, m-k}(x)1_{\frac{1}{2}Q}(x)
		\sbrks{ S^{*,1}\prss{\gamma_g1_{\widetilde{Q}(j)},\, h_Q}(x) - 
			S^{*,1}\prss{\gamma_g1_{\widetilde{Q}(j)},\, h_Q}(x') } 
	}_{=:I_1(x,x')}
	\dx{x}\dx{x'}}\nlspace
&\leq \sum_{P \in B_{f,m-k}}\frac{1}{|P|}  
\int_{P}\int_P
\sign(I_1(x,x')) I_1(x,x')
\dx{x}\dx{x'}\nlspace
&= \sum_{P \in B_{f,m-k}}\frac{1}{|P|}  
\int_{P}\int_{\{x\} - P}
\sign(I_1(x,x-y)) I_1(x,x-y)
\dx{y}\dx{x} \nlspace
&\leq \sum_{P \in B_{f,m-k}}\frac{2^d}{|P_0|}  
\int_{P}\int_{P_0}
\sign(I_1(x,x-y)) I_1(x,x-y)
\dx{y}\dx{x} \nlspace
&= \sum_{P \in B_{f,m-k}}\frac{2^d}{|P_0|}  
\int_{P_0}\int_{P}
\sign(I_1(x,x-y))\prss{\beta_{f, m-k}1_{\frac{1}{2}Q}}(x)
\sbrks{ S^{*,1}\prss{\gamma_g1_{\widetilde{Q}(j)},\, h_Q} - \right.\nlspace
	&\hspace*{10.2cm}\left. \tau_{y}S^{*,1}\prss{\gamma_g1_{\widetilde{Q}(j)},\, h_Q} } (x)
\dx{x}\dx{y}\nlspace
&= \frac{2^d}{|P_0|} 
\int_{P_0}\int_{\eun{d}}
\sign(I_1(x,x-y))\prss{\beta_{f, m-k}1_{\frac{1}{2}Q}}(x)
\sbrks{ S^{*,1}\prss{\gamma_g1_{\widetilde{Q}(j)},\, h_Q} - 
	\tau_{y}S^{*,1}\prss{\gamma_g1_{\widetilde{Q}(j)},\, h_Q} } (x)
\dx{x}\dx{y} \nlspace
&=  \frac{2^d}{|P_0|} 
\int_{P_0}\int_{\eun{d}}
\sbrks{
	\sign(I_1(\cdot,\cdot-y))\prss{\beta_{f, m-k}1_{\frac{1}{2}Q}} \right.\nlspace
	&\hspace*{4cm}\left. - \tau_{-y}\sign(I_1(\cdot,\cdot - y))\prss{\beta_{f, m-k}1_{\frac{1}{2}Q}}
}(x)
S^{*,1}\prss{\gamma_g1_{\widetilde{Q}(j)},\, h_Q}(x)
\dx{x}\dx{y} \nlspace
&= \frac{2^d}{|P_0|} 
\int_{P_0}\int_{\eun{d}}
S\prsm{(I - \tau_{-y})\sign(I_1(\cdot,\cdot-y))\prss{\beta_{f, m-k}1_{\frac{1}{2}Q}},~ \gamma_g1_{\widetilde{Q}(j)} }\cdot h_Q(x)
\dx{x}\dx{y} \nlspace
&\lesssim 2\os{-\eta k}|Q| \angs{\beta_{f, m-k}}_{Q, p}\angs{\gamma_g}_{Q, q}\angs{h_Q}_{Q, r'},\\[-4mm]
\end{align*}
where midway through the calculation we introduce $P_0$, which is a version of the cube $P$ centered at the origin with sidelength $2l(P)$, and  in the last step we've applied the scaled $L^p$-improving continuity estimates for $T_t$. Putting the two estimates together we have
\[
BG \lesssim \angs{g}_{Q_0, q}\sum_{k = 1}^\nf 2\os{-\eta k}\sum_{Q \in \mc{D}_0}|Q| \angs{\beta_{f, m-k}}_{Q, p}\angs{h_Q}_{Q, r'}, 
\]
which is the same point that is arrived at in the proof in \cite{RSS}. One can now use the structure of the functions $\beta_{f, k}$ and the fact that $f,g$ are indicator functions to show that the right-hand side above is 
\[
\lesssim |Q_0|\angs{f}_{Q_0, p}\angs{g}_{Q_0, q}\angs{h}_{Q_0, r'}.
\]
In the next paragraph we show details of this estimate building on the proof in \cite{RSS}. The estimate for the term $GB$ is similar, using $S^{*,2}$ in place of $S^{*,1}$. \\


\textit{\textbf{Details of final estimate:}} This argument is nearly identical to the version of Lacey's argument from \cite{LaceySpherical} that is given in Roncal, Shrivastava, and Shuin \cite{RSS}. From the inequality above for averages of the bad function from the Calder\'on-Zygmund decomposition we have 
\begin{align*}
	BG &\lesssim \angs{g}_{Q_0, q}\sum_{k = 1}^\nf 2\os{-\eta k}\sum_{Q \in \mc{D}_0}|Q| \angs{\beta_{f, m-k}}_{Q, p}\angs{h_Q}_{Q, r'}\nlspace
	&\lesssim \angs{g}_{Q_0, q}\sum_{k = 1}^\nf 2\os{-\eta k}\sum_{Q \in \mc{D}_0}|Q| \angs{1_{F_{m,k}}}_{Q, p}\angs{h_Q}_{Q, r'}\nlspace
		&\hspace*{1.25cm} + \angs{f}_{Q_0, p}\angs{g}_{Q_0, q}\sum_{k = 1}^\nf 2\os{-\eta k}\sum_{Q \in \mc{D}_0}|Q| \angs{1_{E_{m,k}}}_{Q, p}\angs{h_Q}_{Q, r'} =: BG_1 + BG_2. 
\end{align*}
Now we observe that by our restriction $r \geq p ~\Rightarrow~ \frac{1}{p} + \frac{1}{r'} \geq 1$ we can write $\frac{1}{p} - \tau + \frac{1}{r'} = 1$ where $\tau \geq 0$. We define $\frac{1}{\dot{p}} := \frac{1}{p} - \tau$ so that $\dot{p} \geq p$. Now using that $F_{m, k} \se F$ and the fact that $Q \in \mc{D}_0$, we have
\begin{align*}
	\angs{1_{F_{m,k}}}_{Q, p} &= \prs{\frac{1}{|Q|}\int_Q 1^p_{F_{m,k}}}\os{\frac{1}{\dot{p}}}\prs{\frac{1}{|Q|}\int_Q 1^p_{F_{m,k}}}\os{\tau}\nlspace
	&\leq \prs{\frac{1}{|Q|}\int_Q 1^p_{F_{m,k}}}\os{\frac{1}{\dot{p}}}\prs{\frac{1}{|Q|}\int_Q f^p }\os{\tau}\nlspace
	&\lesssim \prs{\frac{1}{|Q|}\int_Q 1^p_{F_{m,k}}}\os{\frac{1}{\dot{p}}}\angs{f}\os{\tau p}_{Q_0, p}. 
\end{align*}
Now to bound the term $BG_1$ we observe essentially as in \cite{RSS} that 
\begin{align*}
	BG_1 &= \angs{g}_{Q_0, q}\sum_{k = 1}^\nf 2\os{-\eta k}\sum_{Q \in \mc{D}_0}|Q| \angs{1_{F_{m,k}}}_{Q, p}\angs{h_Q}_{Q, r'}\nlspace
		&\lesssim \angs{f}\os{\tau p}_{Q_0, p}\angs{g}_{Q_0, q}\sum_{k = 1}^\nf 2\os{-\eta k}\sum_{Q \in \mc{D}_0}|Q| \prs{\frac{1}{|Q|}\int_Q 1^p_{F_{m,k}}}\os{\frac{1}{\dot{p}}} \prs{\frac{1}{|Q|}\int_Q h_Q^{r'}}\os{\frac{1}{r'}}\nlspace
		&\leq \angs{f}\os{\tau p}_{Q_0, p}\angs{g}_{Q_0, q}\sum_{k = 1}^\nf 2\os{-\eta k} 
			\prs{ \sum_{Q\in\mc{D}_0} \int_Q 1_{F_{m = \log_2 l(Q),k}} }\os{\frac{1}{\dot{p}}} 
			\prs{ \sum_{Q \in \mc{D}_0} \int_Q h^{r'} 1_{B_Q} }\os{\frac{1}{r'}}\nlspace
		&\leq \angs{f}\os{\tau p}_{Q_0, p}\angs{g}_{Q_0, q} \angs{f}_{Q_0, p}\os{\frac{p}{\dot{p}}}|Q_0|\os{\frac{1}{\dot{p}}}\angs{h}_{Q_0, r'}|Q_0|\os{\frac{1}{r'}}\nlspace
		&= \angs{f}_{Q_0, p}\angs{g}_{Q_0, q}\angs{h}_{Q_0, r'}|Q_0|,
\end{align*}
since the sets $Q \cap F_{m=\log_2l(Q), k}$ are disjoint as $Q$ varies over $\mc{D}_0$, by the construction of $\beta_f$. 

To bound the term $BG_2$ we use that $\dot{p}\geq p$ and proceed similarly:
\begin{align*}
	BG_2 &= \angs{f}_{Q_0, p}\angs{g}_{Q_0, q}\sum_{k = 1}^\nf 2\os{-\eta k}\sum_{Q \in \mc{D}_0}|Q| \angs{1_{E_{m,k}}}_{Q, p}\angs{h_Q}_{Q, r'}\nlspace
	&\leq \angs{f}_{Q_0, p}\angs{g}_{Q_0, q}\sum_{k = 1}^\nf 2\os{-\eta k}\sum_{Q \in \mc{D}_0}|Q| \angs{1_{E_{m,k}}}_{Q, \dot p}\angs{h_Q}_{Q, r'}\nlspace
	&\leq \angs{f}_{Q_0, p}\angs{g}_{Q_0, q}\sum_{k = 1}^\nf 2\os{-\eta k}
		\prs{ \sum_{Q\in\mc{D}_0} \int_Q 1_{E_{m = \log_2 l(Q),k}} }\os{\frac{1}{\dot{p}}} 
		\prs{ \sum_{Q \in \mc{D}_0} \int_Q h^{r'} 1_{B_Q} }\os{\frac{1}{r'}}\nlspace
	&\lesssim \angs{f}_{Q_0, p}\angs{g}_{Q_0, q}|Q_0|\os{\frac{1}{\dot{p}}}\angs{h}_{Q_0,r'}|Q_0|\os{\frac{1}{r'}}\nlspace
	&= \angs{f}_{Q_0, p}\angs{g}_{Q_0, q}\angs{h}_{Q_0,r'}|Q_0|.
\end{align*} \hfill$\square$\\



\textit{\textbf{Estimate for $BB$:}} Since we have cancellation in both input function in this case, we will essentially use the duality estimate from the $BG$ estimate twice. We have, again with $m = \log_2l(Q)$, 
\begin{align*}
BB &\leq \sum_{Q\in\mc{D}_0}\sum_{k=1}^{\nf}\sum_{k'=1}^\nf 
\underbrace{\abss{\angss{S(\beta_{f, m-k} 1_{\frac{1}{2}Q},\, \beta_{g, m-k'} 1_{\widetilde{Q}(j)}), h_Q}}}_{=:F_{BB}} \nlspace
&= \sum_{Q\in\mc{D}_0}\sum_{k=1}^{3}\sum_{k'=1}^3F_{BB} +  \sum_{Q\in\mc{D}_0}\sum_{k=4}^{\nf}\sum_{k'=1}^3F_{BB}
+ \sum_{Q\in\mc{D}_0}\sum_{k=1}^{3}\sum_{k'=4}^{\nf}F_{BB} + \sum_{Q\in\mc{D}_0}\sum_{k=4}^{\nf}\sum_{k'=4}^{\nf}F_{BB}
\nlspace
&=: I_{BB} + II_{BB} + III_{BB} + IV_{BB}.
\end{align*}
We can estimate the first term using the scaled $L^p$ improving bounds for $T_t$, to get
\[
I_{BB} \lesssim \sum_{k=1}^{3}\sum_{k'=1}^32\os{-\eta(k + k')}\sum_{Q\in\mc{D}_0}|Q|\angs{\beta_{f, m - k}}_{Q, p}\angs{\beta_{g, m - k'}}_{Q, q}\angs{h_Q}_{Q, r'}.
\]
We can use the same procedure that we've used above for the $BG$ and $GB$ terms to get similar estimates for $II_{BB}$ and $III_{BB}$ with one of the summations infinite. Hence it remains to estimate $IV_{BB}$. For this we can apply the duality argument used to estimate $BG$/$GB$ twice -- once for each input:
\begin{align*}
&\abss{\angss{S(\beta_{f,m-k} 1_{\frac{1}{2}Q},\, \beta_{g,m-k'} 1_{\widetilde{Q}(j)}),\, h_Q}} \nlspace
&\leq \sum_{P \in B_{f,m-k}}\abss{\angss{1_P\beta_{f, m - k} 1_{\frac{1}{2}Q},\, S^{*,1}(\beta_{g,m-k'} 1_{\widetilde{Q}(j)},\, h_Q)}}
\nlspace
&\leq \frac{2^d}{|P_0|}\int_{P_0}{\angss{S(
		(I - \tau_{-y_1})\sign(I_1(\cdot,\cdot-y_1))\prss{\beta_{f,m-k}1_{\frac{1}{2}Q}},\,
		\beta_{g,m-k'} 1_{\widetilde{Q}(j)}),
		h_Q}}\dx{y_1}\nlspace
&\leq \frac{2^d}{|P_0|}\int_{P_0} \sum_{P \in B_{2, m - k'}}
\abss{ 
	\angss{1_P\beta_{g,m-k'} 1_{\widetilde{Q}(j)},\,
		S^{*,2}(
		(I - \tau_{-y_1})\sign(I_1(\cdot,\cdot-y_1)){\beta_{f,m-k}1_{\frac{1}{2}Q}},\,
		h_Q)
	}  
}\dx{y_1}\nlspace
&\leq \frac{2^d}{|P_0|}\int_{P_0}\frac{2^d}{|P_1|}\int_{P_1}
\abss{\angss{S\prsm{
		(I - \tau_{-y_1})\sign(I_1(\cdot,\cdot-y_1)){\beta_{f,m-k}1_{\frac{1}{2}Q}},\, \nlspace
		&\hspace*{6cm}(I - \tau_{-y_2})\sign(I_2(\cdot,\cdot-y_2)){\beta_{g,m-k'}1_{\widetilde{Q}(j)}}},\,
		h_Q}}
\dx{y_2}\dx{y_1}\nlspace
&\lesssim 2\os{-\eta(k + k')}|Q|\angs{\beta_{f,m-k}}_{Q, p}\angs{\beta_{g,m-k'}}_{Q, q}\angs{h_Q}_{Q, r'}.
\end{align*}
Putting these estimates together gives
\[
BB \lesssim \sum_{k=1}^{\nf}\sum_{k'=1}^\nf 2\os{-\eta(k + k')}\sum_{Q\in\mc{D}_0}|Q|\angs{\beta_{f, m - k}}_{Q, p}\angs{\beta_{g, m - k'}}_{Q, q}\angs{h_Q}_{Q, r'},
\]
which is equivalent to the estimate that is arrived at in \cite{RSS}. Using the structure of the bad functions $\beta_{f}, \beta_g$ and the fact that $f$ and $g$ are indicator functions, it is shown in \cite{RSS} that the right-hand side is
\[
\lesssim |Q_0|\angs{f}_{Q_0, p}\angs{g}_{Q_0, q}\angs{h}_{Q_0, r'},
\]
which is the estimate asserted in Lemma \ref{lem2}. \\


\textit{\textbf{Details of final estimate:}} The arguments here are again nearly identical to those in \cite{RSS}.  Using the structure of the functions $\beta_f, \beta_g$ we can write
\begin{align*}
	\angs{\beta_{f,m-k}}_{Q,p} &\lesssim \angs{1_{F_{m,k}}}_{Q,p} + \angs{f}_{Q_0, p}\angs{1_{E_{F,m,k}}}_{Q,p},\nlspace
\text{ and } \qquad \angs{\beta_{g,m-k'}}_{Q,q} &\lesssim \angs{1_{G_{m,k'}}}_{Q,q} + \angs{g}_{Q_0, q}\angs{1_{E_{G,m,k'}}}_{Q,q}. 
\end{align*}
and thus
\begin{align*}
	BB &\lesssim \sum_{k=1}^{\nf}\sum_{k'=1}^\nf 2\os{-\eta(k + k')}\sum_{Q\in\mc{D}_0}|Q|
			\angs{1_{F_{m,k}}}_{Q, p}\angs{1_{G_{m,k'}}}_{Q, q}\angs{h_Q}_{Q, r'}
		\nlspace	
		&\hspace*{1.25cm}+ \sum_{k=1}^{\nf}\sum_{j=1}^\nf 2\os{-\eta(k + k')}\sum_{Q\in\mc{D}_0}|Q|
			\angs{1_{F_{m,k}}}_{Q, p}\angs{g}_{Q_0, q}\angs{1_{E_{G,m,k'}}}_{Q, q}\angs{h_Q}_{Q, r'}
		\nlspace
		&\hspace*{1.25cm}+ \sum_{k=1}^{\nf}\sum_{k'=1}^\nf 2\os{-\eta(k + k')}\sum_{Q\in\mc{D}_0}|Q|
		\angs{f}_{Q_0, p}\angs{1_{E_{F,m,k}}}_{Q, p}\angs{1_{G_{m,k'}}}_{Q, q}\angs{h_Q}_{Q, r'}
		\nlspace
		&\hspace*{1.25cm}+ \sum_{k=1}^{\nf}\sum_{k'=1}^\nf 2\os{-\eta(k + k')}\sum_{Q\in\mc{D}_0}|Q|
		\angs{f}_{Q_0, p}\angs{1_{E_{F,m,k}}}_{Q, p}\angs{g}_{Q_0, q}\angs{1_{E_{G,m,k'}}}_{Q, q}\angs{h_Q}_{Q, r'}\nlspace
		&=: BB_1 + BB_2 + BB_3 + BB_4. 
\end{align*} 

Now to estimate $BB_1$, we first consider the case where $\frac{1}{p} + \frac{1}{q} \geq 1$. Then we can write $\frac{1}{p} + \frac{1}{q} = 1 + \tau_1 + \tau_2$ with $\tau_1, \tau_2 \geq 0$ and define $\frac{1}{\dot{p}} := \frac{1}{p} - \tau_1$ and $\frac{1}{\dot{q}} := \frac{1}{q} - \tau_2$. Then we have, using the definitions of $\mc{D}_0$ and the ``bad'' functions similarly to the $BG$ case, 
\begin{align*}
	BB_1 &\lesssim \angs{h}_{Q_0,r'} \sum_{k,k'=1}^\nf 2\os{\eta(k+k')} \sum_{Q\in\mc{D}_0}|Q|\os{1 - \frac{1}{p} - \frac{1}{q}} 
		\prs{ \int_Q 1_{F_{m,k}} }\os{\frac{1}{p}} \prs{ \int_Q 1_{G_{m,k'}} }\os{\frac{1}{q}}\nlspace
	&= \angs{h}_{Q_0,r'} \sum_{k,k'=1}^\nf 2\os{\eta(k+k')} \sum_{Q\in\mc{D}_0}|Q|\os{1 - \frac{1}{p} - \frac{1}{q} + \tau_1 + \tau_2} \prs{ \int_Q 1_{F_{m,k}} }\os{\frac{1}{\dot p}} \prs{ \int_Q 1_{G_{m,k'}} }\os{\frac{1}{\dot q}} 
	\prs{\frac{1}{|Q|}\int_Q f}\os{\tau_1}\prs{\frac{1}{|Q|} \int_Q g }\os{\tau_2}\nlspace
	&\lesssim \angs{f}_{Q_0, p}\os{\tau_1 p}\angs{g}_{Q_0, q}\os{\tau_2 q}\angs{h}_{Q_0,r'} 
		\sum_{k,k'=1}^\nf 2\os{-\eta(k+k')}
		\prs{ \sum_{Q \in \mc{D}_0}\int_Q 1_{F_{m,k}} }\os{\frac{1}{\dot{p}}}\prs{ \sum_{Q \in \mc{D}_0}\int_Q 1_{G_{m,k'}} }\os{\frac{1}{\dot{q}}} \nlspace
	&\lesssim \angs{f}_{Q_0, p}\os{\tau_1 p} \angs{g}_{Q_0, q}\os{\tau_2 q} \angs{h}_{Q_0,r'}  
		\angs{f}_{Q_0, 1}\os{\frac{1}{\dot p}} 
			\angs{g}_{Q_0, 1}\os{\frac{1}{\dot q}} |Q_0|
		\nlspace
	&= \angs{f}_{Q_0, p}\os{\tau_1 p} \angs{g}_{Q_0, q}\os{\tau_2 q} \angs{h}_{Q_0,r'}  
	\angs{f}_{Q_0, p}\os{\frac{p}{\dot p}} 
	\angs{g}_{Q_0, q}\os{\frac{q}{\dot q}} |Q_0|
	\nlspace
	&= |Q_0|\angs{f}_{Q_0, p}\angs{g}_{Q_0, q}\angs{h}_{Q_0, r'},
\end{align*}
relying again essentially on the fact that $f = 1_F$ and $g = 1_G$. In the case where $\frac{1}{p} + \frac{1}{q} < 1$, we notice that $\frac{1}{p} + \frac{1}{q} \geq \frac{1}{r}$ (see \cite{GrafakosM}, Prop.\ 7.15) implies that $\frac{1}{p} + \frac{1}{q} + \frac{1}{r'} \geq 1$. Hence we can choose $\tau_1, \tau_2 \geq 0$ and define $\frac{1}{\dot p} := \frac{1}{p} - \tau_1$ and $\frac{1}{\dot q} := \frac{1}{q} - \tau_2$ so that $\frac{1}{\dot p} + \frac{1}{\dot q} + \frac{1}{r'} = 1$. Then we have via H\"older for $\dot p, \dot q, r'$, using some similar steps as above,
\begin{align*}
	BB_1 &= \sum_{k=1}^{\nf}\sum_{k'=1}^\nf 2\os{-\eta(k + k')}\sum_{Q\in\mc{D}_0}|Q|
	\angs{1_{F_{m,k}}}_{Q, p}\angs{1_{G_{m,k'}}}_{Q, q}\angs{h_Q}_{Q, r'} \nlspace
	&\lesssim \angs{f}_{Q_0, p}\os{\tau_1 p}\angs{g}_{Q_0, q}\os{\tau_2 q} 
		\sum_{k'=1}^\nf 2\os{-\eta(k + k')}\sum_{Q\in\mc{D}_0}
		\prs{ \int_Q 1_{F_{m,k}} }\os{\frac{1}{\dot p}} \prs{ \int_Q 1_{G_{m,k'}} }\os{\frac{1}{\dot q}}
			\prs{ \int_Q h\os{r'}_Q }\os{\frac{1}{r'}} \nlspace
	&\lesssim \angs{f}_{Q_0, p}\os{\tau_1 p}\angs{g}_{Q_0, q}\os{\tau_2 q} 
		\angs{f}_{Q_0, 1}\os{\frac{1}{\dot p}}
		\angs{g}_{Q_0, 1}\os{\frac{1}{\dot q}}
		\angs{h}_{Q_0, 1}\os{\frac{1}{r'}}|Q_0| 
	= \angs{f}_{Q_0, p}\angs{g}_{Q_0, q}\angs{h}_{Q_0, r'}|Q_0|. 
\end{align*}
As the authors in \cite{RSS} note, the bounds in the $\frac{1}{p} + \frac{1}{q} < 1$ case are obtained similarly for the terms $BB_2$, $BB_3$, and $BB_4$, but are even slightly simpler to obtain. 

$BB_2$ and $BB_3$ can both be handled similarly. For the term $BB_3$ when $\frac{1}{p} + \frac{1}{q} \geq 1$, we choose $\tau \geq 0$ and define $\frac{1}{\dot{p}} := \frac{1}{p} - \tau$ so that $\frac{1}{\dot p} + \frac{1}{q} = 1$, and then we have by H\"older twice, 
\begin{align*}
	BB_3 &\lesssim \angs{f}_{Q_0, p}\angs{h}_{Q_0, r'} \sum_{k,k'=1}^\nf 2\os{-\eta(k+k')}\sum_{Q \in \mc{D}_0}|Q|\angs{1_{E_{F, m,k}}}_{Q, \dot p}\angs{1_{G_{m,k'}}}_{Q, q}
	\nlspace
	&= \angs{f}_{Q_0, p}\angs{h}_{Q_0, r'} \sum_{k,k'=1}^\nf 2\os{-\eta(k+k')}\sum_{Q \in \mc{D}_0} 
		\prs{ \int_Q 1_{E_{F,m,k}} }\os{\frac{1}{\dot p}} \prs{ \int_Q 1_{G_{m,k'}} }\os{\frac{1}{q}}
	\nlspace
	&\leq \angs{f}_{Q_0, p}\angs{h}_{Q_0, r'} \sum_{k,k'=1}^\nf 2\os{-\eta(k+k')}
		\prs{ \sum_{Q\in\mc{D}_0} \int_Q 1_{E_{F,m,k}} }\os{\frac{1}{\dot p}}\prs{ \sum_{Q \in \mc{D}_0} \int_Q 1_{G_{m,k'}} }\os{\frac{1}{q}}
	\nlspace
	&\lesssim \angs{f}_{Q_0, p}\angs{h}_{Q_0, r'}|Q_0|\os{\frac{1}{\dot p}}\angs{g}_{Q_0, q}|Q_0|\os{\frac{1}{q}} 
		= \angs{f}_{Q_0, p}\angs{g}_{Q_0, q}\angs{h}_{Q_0, r'}|Q_0|. 
\end{align*} 

Finally, for $BB_4$ in the case where $\frac{1}{p} + \frac{1}{q} \geq 1$ we choose $\tau_1$ and $\tau_2$ as in the $BB_1$ case and we have the simpler inequalities, using $\dot{p} \geq p$ and $\dot{q} \geq q$,
\begin{align*}
	BB_4 &\lesssim \angs{f}_{Q_0, p}\angs{g}_{Q_0, q}\angs{h}_{Q_0, r'} \sum_{k,k'=1}^\nf 2\os{-\eta(k+k')} 
		\sum_{Q\in\mc{D}_0} \angs{1_{E_{F,m,k}}}_{Q,p} \angs{1_{E_{G,m,k'}}}_{Q,q}
	\nlspace
	&\leq \angs{f}_{Q_0, p}\angs{g}_{Q_0, q}\angs{h}_{Q_0, r'} \sum_{k,k'=1}^\nf 2\os{-\eta(k+k')} 
		\prs{ \sum_{Q \in\mc{D}_0} 1_{E_{F,m,k}} }\os{\frac{1}{\dot p}}
		\prs{ \sum_{Q \in\mc{D}_0} 1_{E_{G,m,k'}} }\os{\frac{1}{\dot q}}
	\nlspace
	&\lesssim \angs{f}_{Q_0, p}\angs{g}_{Q_0, q}\angs{h}_{Q_0, r'}|Q_0|\os{\frac{1}{\dot p}}|Q_0|\os{\frac{1}{\dot q}} = \angs{f}_{Q_0, p}\angs{g}_{Q_0, q}\angs{h}_{Q_0, r'}|Q_0|. 
\end{align*} \hfill$\square$\\


\qed 

\section{Continuity Estimates}\label{secContEst}

The proof in Section \ref{secAbsDom}, which is based on the approach introduced by Lacey in \cite{LaceySpherical} and extended to the multilinear setting in \cite{RSS}, utilizes continuity estimates for an operator at each dyadic scale as part of an algorithm to obtain sparse bounds for the corresponding maximal operator. 
In this section we derive continuity estimates for for operators with uniformly decaying bilinear multipliers (with application to the bilinear spherical averages), and the triangle averaging operator, which are used below to apply the abstract sparse domination theorem in Section \ref{secAbsDom} to triangle and bilinear spherical averaging operators. 

\subsection{Uniformly Decaying Bilinear Multipliers}

Grafakos, He, and Slavikova \cite{GHS} prove the following lemma.
\begin{lemma}[\cite{GHS}, Theorem 1.3]\label{lemA}
	Let $1 \leq q < 4$ and set $M_q = \left\lfloor \frac{2d}{4 - q} \right\rfloor + 1$. Let $m(\xi, \eta)$ be a function in $L^q(\eun{2d})\cap C\os{M_q}(\eun{2d})$ satisfying
	\[
	\norms{\pd^\alpha m}_{L^\nf} \leq C_0 < \nf ~ \text{for all $\alpha$ with $|\alpha| \leq M_q$}.
	\] 
	Then there is a constant $C$ depending on $d$ and $q$ such that the bilinear multiplier operator $T_m$ satisfies the bound
	\[
	\norm{T_m}_{L^2\tm L^2 \ra L^1} \leq CC_0\os{1- \frac{q}{4}}\norm{m}_{L^q}\os{\frac{q}{4}}.
	\]
\end{lemma}
\noindent We will use this lemma to obtain continuity estimates for bilinear operators with multiplier $m$ in $L^q$  with $q \in [1, 4)$ and uniform decay. We can then apply these in the particular case when $m = \wh{\mu}$, where $\mu$ is a compactly-supported finite Borel measure. The continuity estimate is:
\begin{lemma}\label{lem7}
	Let $m \in C^\nf(\eun{2d})$ satisfy $m \in L^q(\eun{2d})$ for some $1 \leq q < 4$ and 
	\[
	\abs{\pd^\alpha m(\xi,\eta)} \lesssim_{\alpha}\prs{1 + \abs{(\xi, \eta)}}\os{-s} \quad \text{for some $s > 0$ and each $\alpha$}.
	\]
	Then the bilinear operator $T_m$ with multiplier $m$ satisfies the continuity estimate
	\[
	\norm{T_m(\tau_y f, g)}_{L^1} + \norm{T_m(f, \tau_y g)}_{L^1} \lesssim \abs{y}\os{\frac{s(1 - \frac{q}{4})}{1 + s}}\norm{m}_{L^q}\os{\frac{q}{4}}\norm{f}_{L^2}\norm{g}_{L^2}, 
	\]
	where $\tau_y f(x) := f(x) - f(x-y)$. 
\end{lemma}

\begin{proof}
	Let $\Phi$ be a smooth cutoff function satisfying $1_{[-1 + \epsilon, 1 - \epsilon]^d}\leq \Phi \leq 1_{[-1, 1]^d}$, and define $\Phi_R(\xi) := \Phi(R\fv\xi)$. We will always consider $R \geq 1$, so that the derivatives of $\Phi_R$ of any given order are bounded uniformly independently of $R$. 
	Now we can express, with $\tau_y f(x) = f(x) - f(x - y)$, 
	\begin{align*}
	T_m(\tau_y f,g)(x) &= \int_{\eun{2d}}\wh{f}(\xi)\wh{g}(\eta)(1 - \ef{-\fconst y\cdot \xi})m(\xi,\eta)\ef{\fconst x\cdot(\xi + \eta)}\dx{(\xi,\eta)} = A(x) + B(x),
	\end{align*}
	where
	\begin{align*}
	A(x) := \int_{\eun{2d}}\wh{f}(\xi)\wh{g}(\eta)(1 - \ef{-\fconst y\cdot \xi })m(\xi,\eta)\Phi_R(\xi)\ef{\fconst x\cdot(\xi + \eta)}\dx{(\xi,\eta)},
	\end{align*}
	and
	\begin{align*}
	C(x):= \int_{\eun{2d}}\wh{f}(\xi)\wh{g}(\eta)(1 - \ef{-\fconst y\cdot \xi })m(\xi,\eta)(1 -\Phi_R(\xi))\ef{\fconst x\cdot(\xi + \eta)}\dx{(\xi,\eta)}.
	\end{align*}
	Since $$|\pd^\alpha m(\xi, \eta)| \lesssim_{\alpha} (1 + |(\xi, \eta)|)\os{-s},$$ then the same estimate holds for arbitrary derivatives in $\xi$ or $\eta$ of 
	$$m_C(\xi,\eta):=(1 - \ef{-\fconst y\cdot \xi })m(\xi,\eta)(1 - \Phi_R(\xi)),$$
	since we only consider $|y| \leq 1$ and $R\geq1$. Also note that 
	\[
	\norm{m_C}_{L\os{q}} \leq \norm{m}_{L^q}. 
	\]
	
	Now we will set
	\[
	R:= |y|\os{-a} \quad \text{for some $a \in (0,1)$ to be chosen below.}
	\]
	Then we can take $C_0 \approx |y|\os{as}$ in Lemma \eqref{lemA} to get the bound
	\[
	\norm{C}_{L^2\tm L^2 \ra L^1} \leq C|y|\os{as(1- \frac{q}{4})}\norm{m}_{L^q}\os{\frac{q}{4}}.
	\]
	On the other hand, all derivatives in $\xi$ and $\eta$ of
	\[
	m_A(\xi,\eta):=(1 - \ef{-\fconst y\cdot \xi })m(\xi,\eta)\Phi_R(\xi)
	\]
	satisfy the bound (with implicit constants independent of $y$ here and above)
	\[
	\abs{\pd^\alpha m_A(\xi,\eta)} \lesssim |y|\os{1 - a},
	\]
	so by Lemma \ref{lemA}
	\[
	\norm{A}_{L^2\tm L^2 \ra L^1} \leq C|y|\os{(1-a)(1- \frac{q}{4})}\norm{m}_{L^q}\os{\frac{q}{4}}.
	\]
	Since $|y| \leq 1$, the minimum for the max of both estimates is achieved at $a = (1+s)\fv$, giving best exponent $s(1+s)\fv (1- \frac{q}{4})$ for the factor $|y|$ on the right hand side. 
\end{proof}
\noindent The estimates in the lemma can be interpolated to get simultaneous continuity estimates. 

\subsection{Localized Maximal Operators for Uniformly Decaying Multipliers}
Again assume that $m=\widehat{\mu}$ and
\[
\abs{\pd^\alpha m(\xi,\eta)} \lesssim_{\alpha}\prs{1 + \abs{(\xi, \eta)}}\os{-s} \quad \text{for some $s > 0$ and each $\alpha$}.
\]
We recall that 
\[
L_{\star, t}(f,g)(x) = \sup_{r \in  [t,2t]}\abs{\int_{\sph{2d-1}} f(x - ry)g(x - rz)\dx{\mu(y,z)} } = \sup_{r \in  [t,2t]}\abs{B_r(f,g)(x)},
\]
and observe that by a standard technique (we can assume WLOG $t = 1$ then rescale the resulting estimates for other values of $t$)
\begin{align*}
\norm{L_{\star, 1}(f,g)}_{L^1(\eun{d})} &= \int_{\eun{d}} \sup_{t \in  [1,2]}\abs{ L_1(f,g)(x) + \int_1^s \frac{d}{d t} L_t(f,g)(x) \dx{t}}\dx{x}\nlspace
&\leq \norm{L_1(f,g)}_{L^1(\eun{d})} + \int_{1}^{2}\norm{\frac{d}{dt}L_t(f,g)}_{L^1(\eun{d})}\dx{t}. 
\end{align*}
We can easily deal with the first term as in the previous section, so now we require a uniform estimate for 
\[
\norm{\frac{d}{dt}L_t(f,g)}_{L^1(\eun{d})} \quad \text{for $t \in [1,2]$}. 
\]
For $f, g$ Schwartz,
\begin{align*}
\frac{d}{dt}L_t(f,g)(x) &= \int_{\eun{2d}} \frac{d}{dt}\wh{\mu}(t\xi, t\eta)\wh{f}(\xi)\wh{g}(\eta)\ef{\fconst x\cdot(\xi + \eta)}\dx{(\xi,\eta)},
\end{align*}
and we have for $t \in [1,2]$, 
\[
\abs{ \frac{d}{dt}\wh{\mu}(t\xi, t\eta) } = \abs{(\xi,\eta)\cdot \nabla \wh{\sigma}(t\xi, t\eta)} \lesssim (1 + \abs{(\xi,\eta)})\os{-(s-1))}. 
\]
Now provided $s$ is large enough, this multiplier is still in $L^4(\eun{d})$, and we can again apply the technique from the previous section to get the desired continuity estimate. 

\subsection{Bilinear Spherical Averaging Operators}

Since the surface measure on $\sph{2d-1}$ satisfies
\[
\abss{\wh{\mu}(\xi, \eta)} \lesssim \prs{1 + \abss{(\xi, \eta)}}\os{-\frac{2d-1}{2}},
\]
it follows immediately that $\wh{\mu} \in L^q(\eun{2d})$ for $q > \frac{4d}{2d - 1}$ and $d \geq 2$, and we see that Lemma \ref{lem7} immediately gives continuity estimates for $B_t$ in these dimensions. 
For the single-scale maximal operator $B_{\star, t}$ we can apply the technique above, observing that $\frac{d}{dt}\wh{\sigma}(t\cdot, t\cdot) \in L^q(\eun{2d})$ for $q > \frac{4d}{2d - 3}$. This threshold is $< 4$ for $d \geq 4$, so in these dimensions we obtain the desired continuity estimate.

\subsection{Triangle Averaging Operators}

In \cite{PS} the authors gave one of the first systematic studies of mapping properties of the triangle averaging operator. Although, its symbol is not uniformly decaying, so the results of Grafakos, He, and Slavikova \cite{GHS} can't be used directly, the authors carved up the operator into pieces where the results of Grafakos, He, and Slavikova can be used. In \cite{PS} the estimate on each piece is carefully written down, so one can run exactly the same argument as there, where on each piece we use the same decomposition as was used in the proof of Lemma \ref{lem7}. Keeping track of how the bounds on each piece depend on the translation parameter we obtain the following theorem.

\begin{theorem}
	For $(\frac{1}{p}, \frac{1}{q}, \frac{1}{r})$ in the interior of $\mf{T}^d$, $d\geq 13$, there is an $\eta = \eta(d, p,q,r)$ such that 
	\[
	\norm{T_t\circ ((I - \tau_y)\times I)}_{L^p \tm L^q \ra L^r} + \norm{T_t\circ (I \times (I - \tau_y))}_{L^p \tm L^q \ra L^r} 
	\lesssim t\os{d\prs{\frac{1}{r} - \frac{1}{p} - \frac{1}{q}}}\prs{\frac{\abs{y}}{t}}\os{\eta}.
	\]
\end{theorem}
\noindent Finally, we can use interpolation, such as in \cite[Theorem 7.2.2]{GrafakosM}, to get simultaneous continuity estimates:
\begin{theorem}
	For $(\frac{1}{p}, \frac{1}{q}, \frac{1}{r})$ in the interior of $\mf{T}^d$, $d\geq 13$, there exist $\eta_1, \eta_2$ such that
	\[
	\norm{T_t\circ ((I - \tau_{y_1})\times (I - \tau_{y_2}))}_{L^p \tm L^q \ra L^r} \lesssim t\os{d\prs{\frac{1}{r} - \frac{1}{p} - \frac{1}{q}}}\prs{\frac{\abs{y_1}}{t}}\os{\eta_1}\prs{\frac{\abs{y_2}}{t}}\os{\eta_2}.
	\]
\end{theorem}

We note that we have no continuity estimates for the single scale maximal triangle averaging operator, as we don't have an analogous proof of its boundedness. There do exist bounds for the single scale maximal triangle averaging operator, that depend on the simple bounds 
\[
\sup_{t\in[R, 2R]}|T_t(f,g)(x)| \leq \min\cbrksm{ \norm{f}_{L^\nf}\cdot\sup_{t\in[R, 2R]}\mc{A}_t(|g|),~ \norm{g}_{L^\nf}\cdot\sup_{t\in[R, 2R]}\mc{A}_t(|f|) }
\]
or variants thereof. The reader might be tempted to apply the continuity estimates of Lacey \cite{LaceySpherical} for the single scale spherical maximal operator, but we warn the reader that having the absolute value inside the averaging operator destroys the delicate, Fourier analytic proof of Lacey.

\section{Application to Triangle and Bilinear Spherical Averages}
	
\subsection{Sparse Bounds for Lacunary Operators}
	
A direct application of the Theorem \ref{thm1} using the continuity estimates shown in the previous section yields the sparse bounds for the lacunary operators $\tlac$ and $\blac$ in Theorem \ref{thm0}.
	
\subsection{Sparse Bounds for $\tful$, $\bful$, and $\lful$}\label{secFullMax}
	
In this section we describe how the proof of the abstract sparse domination theorem in Section \ref{secAbsDom} can be extended to obtain bounds for the full maximal triangle averaging operator
\[
\tful(f,g) := \sup_{t > 0} |T_t(f,g)(x)|. 
\]
assuming the appropriate continuity estimates. Essentially the same arguments will apply to $\lful$, as long as the required $L^p$-improving and continuity estimates are available along with some bounds for the corresponding linearized operator introduced below. In particular, this argument also applies to obtain sparse bounds for $\bful$. 

We recall the single-scale maximal operator
\[
T_{\star, t}(f,g)(x) = \sup_{s\in [t, 2t]}|T_s(f,g)(x)|, 
\]
and observe that
\[
\tful(f,g)(x) = \sup_{j \in \znm}T_{\star, 2^j}(f,g)(x). 
\]
Then by an argument parallel to that in the first part of the proof of Theorem \ref{thm1} in Subsection~\ref{subsecLem1}, using the sublinearity and support properties of $T_{\star, t}$ we have the pointwise domination
\begin{align*}
	\tful(f,g)(x) &\leq \sum_{i,j=1}^{3^d}\mc{M}_{\star, i,j}(f,g)(x),
\end{align*}
where for $l_Q = q$ we define
\begin{align*}
	\mc{M}_{\star, i,j}(f,g)(x) &:= \sup_{Q \in \mc{D}^i}T_{\star, 2\os{q-4}}(1_{\frac{1}{3}Q}f,\, 1_{(\frac{1}{3}Q)(j)}g)(x)\nlspace
	&=:\sup_{Q \in \mc{D}^i} T_{\star, Q}^j(f,g)(x). 
\end{align*}
Here we use the same notation as in Subsection~\ref{subsecLem1} for cubes and lattices. The operators $T_{\star,Q}^j$ are not linear but sublinear. 

Many of the arguments given for $\tlac$ carry over to sublinear operators with no essential changes. However the duality arguments used in Subsection \ref{subsecLem2} to prove Lemma \ref{lem2} makes use of adjoint operators. For the operator $T_t$ the adjoint operators are easily derived, but that doesn't appear to be the case for $T_{\star, t}$, due in part to its nonlinearity. However, we can linearize the operator $T_{\star,1}$ as follows, for $f,g \geq 0$,
\begin{align*}
T_{\star,1}(f,g)(x) &= \sup_{\substack{\text{$t$ measurable},\\ 1\leq t(x)\leq 2}}{ \int_{  \mf{T} } f(x - t(x)y)g(x - t(x)z)\dx{\mu(y,z)} }\nlspace
	&=:\sup_{\substack{\text{$t$ measurable},\\ 1\leq t(x)\leq 2}} T_{t(x)}(f,g)(x). 
\end{align*}
Now we can take the $L^\nf$ norm of one function and dominate by a unit-scale spherical maximal function, and by interpolation we will get bounds for the operators $T_{t(x)}$,
\[
T_{t(x)} \cl L^p \tm L^q \ra L^r \quad \text{ for a triple $(p,q,r)$ with $r >1$.  }
\]
Then for a fixed $g \in L^q$ and $h \in L^{r'}$ we have
\[
|L_{g, h}(f)| := |\angs{T_{t(x)}(f,g), h}| \lesssim_{t(x)} \prs{\norm{g}_{L^q}\norm{h}_{L\os{r'}}} \norm{f}_{L^p}. 
\]
Hence by $L^p$ duality there is a function $T^{*,1}_{t(x)}(g,h) \in L\os{p'}$ such that 
\[
L_{g,h}(f) = \angss{f, T^{*,1}_{t(x)}(g,h)}, \quad \text{with} \quad \norms{ T^{*,1}_{t(x)}(g,h) }_{L\os{p'}} \lesssim \norm{g}_{L^q}\norm{h}_{L\os{r'}}. 
\]
Similarly we can obtain a second adjoint $T^{*,2}_{t(x)}(f,h)$. This same procedure will apply more generally to $L_{\star, t}$ whenever the maximal operators
\[
\sup_{s \in [t, 2t]}\abs{\int f(x - sy)\dx{\mu(y,z)}} \qquad \text{and} \qquad \sup_{s \in [t, 2t]}\abs{\int g(x - sz)\dx{\mu(y,z)}}
\]
satisfy some $L^p$ bounds. This will be the case whenever the multipliers $\wh{\mu}$ have some uniform decay, by the technique of Rubio de Francia \cite{Rubio}. This applies in particular to $\bful$, as in \cite{BarrionuevoEA}.

We can use these adjoints whenever $f,g,h$ are nice functions.  Then by Fatou's lemma (and taking an approximating sequence $t_n(x)$,
for nice functions $f,g,h$), where we write $m=\log_2(l(Q))$, we have
\begin{align*}
	&\abss{ \angss{ T_{\star, 2^{m-4}} (1_{\frac{1}{3}Q}\beta_{f, q - k},\, 1_{(\frac{1}{3}Q)(j)}\gamma_g), h_Q } }\nlspace
	&\leq \sup_{t(x)}\abss{ \angss{ T_{2^{m-4}t(x)} (1_{\frac{1}{3}Q}\beta_{f, q - k},\, 1_{(\frac{1}{3}Q)(j)}\gamma_g), h_Q } }\nlspace
	&\leq \sup_{t(x)}\sum_{P \in B_{f,q-k}}
		\abss{ \angss{ T_{2^{m-4}t(x)} (1_P1_{\frac{1}{3}Q}\beta_{f, q - k},\, 1_{(\frac{1}{3}Q)(j)}\gamma_g), h_Q } } \nlspace
	&\leq \sup_{t(x)}\frac{2^d}{|P_0|}\int_{P_0}
	\abss{ \angss{ T_{ 2^{m-4}t(x)} \prsm{(I - \tau_{-y})\sign(I_1(\cdot,\cdot-y))\prss{1_{\frac{1}{3}Q}\beta_{f, q - k}},\, 1_{(\frac{1}{3}Q)(j)}\gamma_g}, h_Q } } \dx{y} \nlspace
	&\leq \frac{2^d}{|P_0|}\int_{P_0}
	{ \angss{ {T_{\star, 2^{m-4}} \prss{(I - \tau_{-y})\sign(I_1(\cdot,\cdot-y))\prss{1_{\frac{1}{3}Q}\beta_{f, q - k}},\, 1_{(\frac{1}{3}Q)(j)}\gamma_g}}, h_Q } } \dx{y} \nlspace
%
	&\lesssim 2\os{-\eta k}|Q| \angs{\beta_{f, k}}_{Q, p}\angs{\gamma_g}_{Q, q}\angs{h_Q}_{Q, r'},
\end{align*}
for $(1/p,1/q,1/r)$ in the interior of the region where $L^p$-improving continuity estimates for $T_{\star, t}$ hold. 
We can carry out a similar argument in the ``bad-bad'' case. This shows that we can extend the arguments given for $\tlac$ to apply to $\tful$, and more generally from $\llac$ to $\lful$.
	
\section{Weighted Bounds}\label{secWeight}
	
As shown in \cite{RSS} for product-type spherical averages, the sparse bounds in Theorem \ref{thm1} imply a range of weighted bounds for the operator $\tlac$ for multilinear Muckenhoupt weights. Even though the duality technique used to prove Theorem \ref{thm1} required that we restrict to bounds with $r > 1$, we can use a combination of $L^p$ improving bounds and limited-range extrapolation to get weighted bounds in the quasi-Banach range. The result is:
\begin{corollary}\label{two}
	Let $(1/r_1, 1/r_2, 1/r_3)$ be an $L^p$-improving triple in the interior of $\mf{T}^d$, $d\geq 13$, with $r_3 > 1$. Then for any $q_1 > r_1$ and $q_2 > r_2$, with $\frac{1}{q} := \frac{1}{q_1}+\frac{1}{q_2}$ we have 
	\begin{equation*}
	\norms{\tlac(f,g)}_{L\os{q}(v)} \lesssim \norms{f}_{L\os{q_1}(v_1)}\norms{g}_{L\os{q_2}(v_2)}
	\end{equation*}
	for all $f \in L\os{q_1}(v_1)$ and $g \in L\os{q_2}(v_2)$, for all weights $(v_1, v_2) \in A_{\vec{q}, \vec{r}}$ with $v := \prod_{i=1}^2v_i\os{\frac{q}{q_i}}$. 
	
	Let $(1/r_1, 1/r_2, 1/r_3)$ be an $L^p$-improving triple in the interior of $\mf{S}^d$, $d\geq 2$, with $r_3 > 1$. Then for any $q_1 > r_1$ and $q_2 > r_2$, with $\frac{1}{q} := \frac{1}{q_1}+\frac{1}{q_2}$ we have 
	\begin{equation*}
	\norms{\blac(f,g)}_{L\os{q}(v)} \lesssim \norms{f}_{L\os{q_1}(v_1)}\norms{g}_{L\os{q_2}(v_2)}
	\end{equation*}
	for all $f \in L\os{q_1}(v_1)$ and $g \in L\os{q_2}(v_2)$, for all weights $(v_1, v_2) \in A_{\vec{q}, \vec{r}}$ with $v := \prod_{i=1}^2v_i\os{\frac{q}{q_i}}$. 
	
	Let $(1/r_1, 1/r_2, 1/r_3)$ be an $L^p$-improving triple in the interior of $\mf{S}_{\star}^d$, $d\geq 4$, with $r_3 > 1$. Then for any $q_1 > r_1$ and $q_2 > r_2$, with $\frac{1}{q} := \frac{1}{q_1}+\frac{1}{q_2}$ we have 
	\begin{equation*}
	\norms{\bful(f,g)}_{L\os{q}(v)} \lesssim \norms{f}_{L\os{q_1}(v_1)}\norms{g}_{L\os{q_2}(v_2)}
	\end{equation*}
	for all $f \in L\os{q_1}(v_1)$ and $g \in L\os{q_2}(v_2)$, for all weights $(v_1, v_2) \in A_{\vec{q}, \vec{r}}$ with $v := \prod_{i=1}^2v_i\os{\frac{q}{q_i}}$. 
\end{corollary}

As in \cite{RSS}, this corollary follows by a result of Li, Martell, and Ombrosi \cite{LMO} on extrapolation for multilinear Muckenhoupt classes. For locally integrable functions $0 < w_i < \nf$ we say that $\vec{w} = (w_1, \ldots, w_m) \in A_{\vec{p}, \vec{r}}$ provided that
\[
[\vec{w}]_{A_{\vec{p}, \vec{r}}} := \sup_Q \prs{ \dashint_Q w\os{\frac{r'_{m+1}}{r'_{m+1} - p}}\dx{x} }\os{\frac{1}{p} - \frac{1}{r'_{m+1}}}\prod_{i=1}^{m}\prs{ \dashint_Q w_i\os{\frac{r_i}{r_i - p}}\dx{x} }\os{\frac{1}{r_i} - \frac{1}{p}} < \nf,
\]
where
\[
\frac{1}{p} := \frac{1}{p_1} + \cdots + \frac{1}{p_m} \quad \text{and} \quad w := \prod_{i=1}^m w_i\os{\frac{p}{p_i}}. 
\]
Then we have the extrapolation result:
\begin{theorem}[\cite{LMO}, Corollary 2.15]\label{thm2}
	Fix $\vec{r} = (r_1, \ldots, r_{m+1})$ with $r_i \geq 1$ for $1 \leq i \leq m+1$ and $\sum_{i=1}^{m+1}\frac{1}{r_i} > 1$, and a sparsity constant $\gamma \in (0, 1)$. Let $T$ be an operator so that for every $f_1, \ldots, f_m, h \in C_c^\nf(\eun{d})$, with the supremum over sparse families,
	\[
	\angss{|T(f_1, \ldots, f_m)|, h} \lesssim \sup_{\mc{S}}\sum_{Q \in \mc{S}}|Q|\angs{h}_{Q, r_{m+1}}\prod_{i=1}^m \angs{f_i}_{Q, r_i}. 
	\]
	Then for all exponents $\vec{q} = (q_1, \ldots, q_m)$ with $\vec{r} \prec \vec{q}$, all weights $\vec{v} = (v_1, \ldots, v_m) \in A_{\vec{q}, \vec{r}}$, and all $f_1, \ldots, f_m \in C^\nf_c(\eun{d})$ we have 
	\[
	\norms{T(f_1, \ldots, f_m)}_{L^q(v)} \lesssim \prod_{i=1}^m \norms{f_i}_{L\os{q_i}(v_i)},
	\]
	where
	\[
	\frac{1}{q} := \frac{1}{q_1} + \cdots + \frac{1}{q_m} \quad \text{and} \quad v:= \prod_{i=1}^m v_i\os{\frac{q}{q_i}}. 
	\]
\end{theorem}
\noindent To use this result to obtain Corollary \ref{two}, we choose an $L^p$-improving triple $(1/r_1, 1/r_2, 1/r'_3)$ in the interior of $\mf{B}$ with $r_3' > 1$. Then 
\[
\frac{1}{r_1} + \frac{1}{r_2} + \frac{1}{r_3} > 1,
\]
so for any $q_1 > r_1$ and $q_2 > r_2$ and $\frac{1}{q} := \frac{1}{q_1}+\frac{1}{q_2}$, application of Theorem \ref{thm1} and Theorem \ref{thm2} yields
\begin{equation}\label{eq2}
\norms{\tlac(f,g)}_{L\os{q}(v)} \lesssim \norms{f}_{L\os{q_1}(v_1)}\norms{g}_{L\os{q_2}(v_2)}
\end{equation}
for all $f \in L\os{q_1}(v_1)$ and $g \in L\os{q_2}(v_2)$ and all weights $v, v_1, v_2$ as in Theorem \ref{thm2}. The same argument applies to bounds for $\blac$ and $\bful$.

\end{document}